\documentclass[12pt]{article}
\usepackage{latexsym, amsmath, amsfonts, amsthm, amssymb, a4wide, url}%, showlabels}

\def \RR {\mathbb R}

\def \Tr {\text{\rm Tr}}

\def \i {\rm Int}
\def \a {\rm Aut}

\def \b {\rm b}
\def \cov {\rm Cov}

\def \EE {\mathbb E}

\def \CC {\mathbb C}
\def \PP {\mathbb P}

\def \eps {\varepsilon}

\newtheorem{theorem}{Theorem}[section]
\newtheorem{lemma}[theorem]{Lemma}
\newtheorem{proposition}[theorem]{Proposition}
\newtheorem{corollary}[theorem]{Corollary}
\newtheorem{remark}[theorem]{Remark}

\def\myffrac#1#2 in #3{\raise 2.6pt\hbox{$#3 #1$}\mkern-1.5mu\raise 0.8pt\hbox{$
#3/$}\mkern-1.1mu\lower 1.5pt\hbox{$#3 #2$}}

\def\qed{\hfill $\vcenter{\hrule height .3mm
\hbox {\vrule width .3mm height 2.1mm \kern 2mm \vrule width .3mm
height 2.1mm} \hrule height .3mm}$ \bigskip}

\def \tr {{\rm Tr}}

\begin{document}

\title{Isotropic constants and Mahler volumes}
\author{Bo'az Klartag}
\date{}
\maketitle

\begin{abstract}
This paper contains a number of results related to volumes of projective perturbations of convex bodies
and the Laplace transform on convex cones. First, it is shown that a sharp version of Bourgain's
slicing conjecture implies the Mahler conjecture for convex bodies that are not necessarily centrally-symmetric. Second,
we find that by slightly translating the polar of a centered convex body, we may obtain another body
with a bounded isotropic constant. Third, we provide a counter-example to a conjecture by Kuperberg
on the distribution of volume in a body and in its polar.
\end{abstract}

\section{Introduction}

This paper describes interrelations between duality and distribution of volume
in convex bodies. A convex body is a compact, convex subset $K \subseteq \RR^n$ whose interior $\i(K)$ is non-empty.
If $0 \in \i(K)$, then the  polar body is defined by
$$ K^{\circ} = \{ y \in \RR^n \, ; \, \forall x \in K, \ \langle x,y \rangle \leq 1 \}. $$
The polar body  $K^{\circ}$ is itself a convex body with the origin in its interior, and moreover $(K^{\circ} )^{\circ} = K$.
The {\it Mahler volume} of a convex body $K \subseteq \RR^n$ with the origin in its interior is defined as
$$ s(K) = Vol_n(K) \cdot Vol_n(K^{\circ}), $$
where $Vol_n$ is $n$-dimensional volume. In the class of convex bodies with barycenter at the origin,
the Mahler volume is maximized for ellipsoids, as proven by Santal\'o \cite{Sa}, see also Meyer and Pajor \cite{MP}. The
Mahler conjecture suggests that for any convex body $K \subseteq \RR^n$ containing the origin in its interior,
\begin{equation}
s(K) \geq s(\Delta^n) = \frac{(n+1)^{n+1}}{(n!)^2}, \label{eq_1001}
\end{equation}
where $\Delta^n \subseteq \RR^n$ is any simplex whose vertices span $\RR^n$ and add up to zero. The conjecture
was verified for convex bodies with certain symmetries in the works of Barthe and Fradelizi \cite{BF},
Kuperberg \cite{Ku} and Saint Raymond \cite{sara}. In two dimensions the conjecture was proven by Mahler \cite{Ma}, see
also Meyer \cite{Me}. There is also a well-known version
of the Mahler conjecture for centrally-symmetric convex bodies (i.e., when $K = -K$) that will not be  discussed much here.
It was proven by Bourgain and Milman \cite{BM} that
for any convex body $K \subseteq \RR^n$ with the origin in its interior,
\begin{equation} s(K) \geq c^n \cdot s(\Delta^n) \label{eq_531} \end{equation} for some universal constant
$c > 0$. There are several beautiful, completely different proofs of the Bourgain-Milman inequality
in addition to the original argument, including proofs  by Kuperberg \cite{Ku}, by Nazarov \cite{Na} and by Giannopoulos,
Paouris and Vritsiou \cite{GPV}.
 The covariance matrix of a convex body $K \subseteq \RR^n$ is the matrix $\cov(K) = (\cov_{ij})_{i,j=1,\ldots,n}$
where
$$ \cov_{ij} = \frac{\int_K x_i x_j dx}{Vol_n(K)}  \, - \, \frac{\int_K x_i}{Vol_n(K)} \cdot  \frac{\int_K x_j}{Vol_n(K)}. $$
The isotropic constant of a convex body $K \subseteq \RR^n$ is the parameter $L_K > 0$ defined via
\begin{equation}  L_K^{2n} = \frac{\det \cov(K)}{Vol_n(K)^2}. \label{eq_212} \end{equation}
We equip the space of convex bodies in $\RR^n$ with the usual Hausdorff topology. The Mahler volume is a continuous function defined
on the class of convex bodies in $\RR^n$ containing the origin in their interior. A standard compactness argument shows that the minimum of the Mahler volume in this class is indeed attained. Below we present  a variational argument in the class of projective images of $K$
that yields the following:

\begin{theorem} Let $K \subseteq \RR^n$ be a convex body which is a local minimizer of the Mahler volume.
	Then $\cov(K^{\circ}) \geq (n+2)^{-2} \cdot \cov(K)^{-1}$ in the sense of symmetric matrices, and
	\begin{equation} L_K \cdot L_{K^{\circ}} \cdot s(K)^{1/n} \geq \frac{1}{n+2}.  \label{eq_947}
	\end{equation} \label{thm_1151}
	Consequently any global minimizer must satisfy $L_K \geq L_{\Delta^n}$ or $L_{K^{\circ}} \geq L_{\Delta^n}$.
 \label{thm_206}
\end{theorem}

It is well-known that $L_K > c$ for any convex body $K \subseteq \RR^n$, where $c > 0$ is a universal constant. In fact, the minimal isotropic constant is attained
for ellipsoids. Bourgain's slicing problem \cite{Bou1, Bou2} asks whether $L_K < C$ for a universal constant
$C > 0$. The slicing conjecture has several equivalent formulations,
and it is related to quite a few asymptotic questions about the distribution of volume in high-dimensional convex bodies.
Currently the best known estimate is $L_K \leq C n^{1/4}$ for a convex body $K \subseteq \RR^n$. This was shown in
\cite{K_quarter}, slightly improving upon an earlier estimate by Bourgain \cite{Bou3, Bou4}.
Two sources of information on the slicing problem are the recent book by
Brazitikos, Giannopoulos, Valettas and Vritsiou  \cite{BGVV} and the survey paper by Milman and Pajor \cite{MiP}.
A strong version of the slicing conjecture proposes that for any convex body $K \subseteq \RR^n$,
\begin{equation}
L_K \leq L_{\Delta^n} = \frac{(n!)^{\frac 1 n}}{(n+1)^{\frac{n+1}{2n}} \cdot \sqrt{n+2}}.
\label{eq_1148}
\end{equation}
This conjecture holds true in two dimensions. See also Rademacher \cite{Rad} for supporting evidence.
Theorem \ref{thm_1151} admits the following:

\begin{corollary} The strong version (\ref{eq_1148}) of Bourgain's slicing
	conjecture implies 	Mahler's conjecture (\ref{eq_1001}). \label{cor_1213}
\end{corollary}

In order to see why inequality (\ref{eq_947})  implies Corollary \ref{cor_1213}, observe that by (\ref{eq_947}) and (\ref{eq_1148}), for any local minimizer $K \subseteq \RR^n$ of the Mahler volume,
$$ \frac{1}{(n+2)^n} \leq L_K^n \cdot L_{K^{\circ}}^n \cdot s(K) \leq
L_{\Delta^n}^{2n} \cdot s(K) =  \frac{(n!)^{2}}{(n+1)^{n+1} \cdot (n+2)^{n}} \cdot s(K),
$$
which clearly yields (\ref{eq_1001}). We are aware of two more conditional statements  in the spirit
of Corollary \ref{cor_1213}. Artstein-Avidan, Karasev and Ostrover \cite{AKO} proved that the Mahler conjecture for centrally-symmetric bodies would follow
from the Viterbo conjecture in symplectic geometry.
It was recently shown that  the Minkowski conjecture would follow from a strong version of the
centrally-symmetric slicing conjecture, see Magazinov \cite{Mag}.

\smallskip
For the proof of Theorem \ref{thm_206}
we use the Laplace transform in order to analyze the Mahler volume in the space of projective images of $K$. The Laplace transform was also used
in \cite{K_quarter} in order to prove the isomorphic version of the slicing problem. Here we observe the following
variant of the result from \cite{K_quarter}:

\begin{theorem} Let $K \subseteq \RR^n$ be a convex body with barycenter at the origin and let $0 < \eps < 1/2$. Then there exists a convex
	body $T \subseteq \RR^n$ such that the following hold:
	\begin{enumerate}
		\item[(i)] $\displaystyle (1 - \eps) K \subseteq T \subseteq (1 + \eps) K$.
		\item[(ii)] The polar body $T^{\circ}$ is a translate of $K^{\circ}$, i.e., $T^{\circ} = K^{\circ} - y$ for some $y \in \i(K^{\circ})$.
		\item[(iii)] $\displaystyle L_T < C / \sqrt{\eps}$, where $C > 0$ is a universal constant.
	\end{enumerate}
\label{thm_903}
\end{theorem}

We say that two convex sets $K_1 \subseteq \RR^{n_1}$ and $K_2 \subseteq \RR^{n_2}$ are {\it affinely equivalent}
if there exists an affine map $T: \RR^{n_1} \rightarrow \RR^{n_2}$, one-to-one on the set $K_1$, with $K_2 = T(K_1)$.
 It is a curious fact that  the quantity $$ L_K^2 \cdot s(K)^{1/n} $$ attains the same value $1/(n+2)$ when $K \subseteq \RR^n$ is an ellipsoid and when
$K \subseteq \RR^n$ is a simplex, see Alonso--Guti\'errez \cite{Alonso}, where in these two examples and also in the next one we assume that the barycenter of $K$ lies at the origin. Intriguingly,
\begin{equation} L_K^2 \cdot s(K)^{1/n} = 1 / (n+2)
\label{eq_1040} \end{equation}  also when $n = \ell (\ell+1) / 2 - 1$ and $K \subseteq \RR^n$ is affinely equivalent to the collection
of all symmetric, positive semi-definite $\ell \times \ell$ matrices of trace one. This is not a mere coincidence.
A common feature amongst these three examples is that they are hyperplane sections of {\it convex homogeneous cones}.

\smallskip
A convex cone in $\RR^{n+1}$ is a convex subset $V$ such that $\lambda x \in V$ for any $x \in V$ and $\lambda \geq 0$.
We say that a convex cone $V \subseteq \RR^{n+1}$ is proper if it is closed, has a non-empty interior, and does not contain a full line.
A convex cone $V \subseteq \RR^{n+1}$ is
 {\it homogeneous} if for any $x, y \in \i(V)$ there is an invertible,
linear transformation $T: \RR^{n+1} \rightarrow \RR^{n+1}$ with $T(x) = y$ and $$ T(V) = V. $$
Recall that the Santal\'o point of a convex body $K \subseteq \RR^n$ is the unique point $x_0$ in the interior of $K$
such that the barycenter of $(K - x_0)^{\circ}$ lies at the origin. See, e.g., Schneider \cite[Section 10.5]{sch} for information about the Santal\'o point.

\begin{theorem} Let $K \subseteq \RR^n$ be a convex body that is affinely equivalent to a hyperplane section of a
	convex, homogeneous cone in $\RR^{n+1}$. Then,
	\begin{enumerate}
		\item[(i)] The barycenter of $K$ is its Santal\'o point.
		\item[(ii)] If the barycenter of $K$ lies at the origin, then
		\begin{equation} \cov(K^{\circ}) = (n+2)^{-2} \cdot \cov(K)^{-1},
		\label{eq_1030} \end{equation}
		and consequently $L_K \cdot L_{K^{\circ}} \cdot s(K)^{1/n} = 1 / (n+2)$.
	\end{enumerate}
	\label{prop_356}
\end{theorem}

Thus, for instance, (\ref{eq_1040}) also holds true  for the convex set
that consists of all positive-definite Hermitian or quaternionic-Hermitian matrices  of trace one. In all of the examples above (ball, simplex and convex collections of matrices) the convex body $K^{\circ}$ has turned out to be a linear image of $K$ as the cone is self-dual, thus $L_K = L_K^{\circ}$. There exist homogeneous cones that are not self-dual (see, e.g., Vinberg \cite{Vin} and references therein),
yet we see from Theorem \ref{prop_356} that the barycenters and covariance matrices
of hyperplane sections of homogeneous cones automatically satisfy a certain duality property.
Similar duality properties apply for higher moments as well.

\smallskip The relation (\ref{eq_1030}) between the covariance of a convex body and the covariance of the polar body reminds us of the quantity
\begin{equation}
\phi(K) = \EE\langle X, Y \rangle^2 = \frac{1}{s(K)} \int_K \int_{K^{\circ}} \langle x, y \rangle^2 dx dy \label{eq_226} \end{equation}
considered by Kuperberg \cite{Ku}. Here, $X$ and $Y$ are independent random vectors, $X$ is uniformly distributed  in the convex body $K$,
while $Y$ is uniformly distributed  in $K^{\circ}$.
Assume that $K \subseteq \RR^n$ satisfies the assumptions of Theorem \ref{prop_356}(ii). Then
by the conclusion of Theorem \ref{prop_356},
$$ \phi(K) = \Tr[ \cov(K) \cdot \cov(K^{\circ})] = n / (n+2)^2. $$
It was shown by Kuperberg \cite{Kuper_private} that the Euclidean ball is a local maximizer of the functional  $K \mapsto \phi(K)$ in the class of centrally-symmetric convex bodies in $\RR^n$ with a $C^2$-smooth boundary endowed with the natural topology.
Conjecture 5.1 in \cite{Ku} suggests that this local maximum is in fact a global one, i.e., that  $$ \phi(K) \leq n / (n+2)^2 $$
for any centrally-symmetric convex body $K \subseteq \RR^n$.
This conjecture was verified
in the case where $K$ is the unit ball of $\ell_p^n$ for $1 \leq p \leq \infty$, see Alonso-Guti\'errez \cite{Alonso}.
Relations of $\phi(K)$ to the slicing problem were observed by Giannopoulos in \cite{Gian_phd}, where it was shown that $\phi(K) \leq C / n$ when $K \subseteq \RR^n$ is a body of revolution and $C > 0$ is a universal constant.

\smallskip
Nevertheless, Kuperberg described Conjecture 5.1 in \cite{Ku} as ``perhaps less likely''.
Our next proposition shows that  this conjecture is indeed false in a sufficiently high dimension. A convex body $K \subseteq \RR^n$ is {\it unconditional} if for any $x \in \RR^n$,
$$ x = (x_1,\ldots,x_n) \in K \qquad \Longleftrightarrow \qquad (|x_1|, \ldots, |x_n|) \in K. $$

\begin{proposition} For any $n \geq 1$ there exists an unconditional, convex body $K \subseteq \RR^n$ with
\begin{equation} \frac{1}{s(K)} \cdot \int_K x_1^2 dx \cdot \int_{K^{\circ}} x_1^2 dx \geq c, \label{eq_425}
\end{equation}
	where $c > 0$ is a universal constant. In particular, $\phi(K) \geq c$.	
	\label{prop_1157}
\end{proposition}

The example of
Proposition \ref{prop_1157} is optimal up to a universal constant as clearly $\phi(K) \leq 1$ for any centrally-symmetric convex body $K$ in any dimension.
We say that $K \subseteq \RR^n$ is $1$-symmetric
or that it has the symmetries of the cube if $K$ is unconditional and furthermore for any permutation $\sigma \in S_n$ and a point $x \in \RR^n$,
$$ x = (x_1,\ldots,x_n) \in K \qquad \Longleftrightarrow \qquad (x_{\sigma(1)}, \ldots, x_{\sigma(n)}) \in K. $$
When the convex body
$K \subseteq \RR^n$ has the symmetries of the cube,
 not only do we know that $\EE \langle X, Y \rangle^2 \leq C / n$,
but we may also prove that the random variable $\langle X, Y \rangle$ is approximately Gaussian when the dimension $n$ is large. This
is a corollary of the results of \cite{k_ptrf}.
Thus, perhaps Kuperberg's conjecture or even the stronger versions from \cite{Ku} hold true in the case of $1$-symmetric convex bodies.

\smallskip Theorem 1.1 and Theorem 1.4 are proven in Sections 2 and 3. Section 4 discusses some examples,
while Proposition 1.5 is proven in Section 5. In Sections 6 and 7 we prove Theorem 1.3.
We continue with some notation and conventions that will be used below.
The relative interior of a convex set $K \subseteq X$ is its interior relative to the affine subspace
spanned by $K$. We abbreviate $A + B = \{ a + b \, ; \, a \in A, b \in B \}$
and $x + A = \{ x + a \, ; \, x \in A \}$. We write $\langle x, y \rangle$ or $x \cdot y$
for the standard scalar product of $x,y \in \RR^n$, and $|x| = \sqrt{\langle x, x \rangle}$.
We denote by $A^* \in \RR^{m \times n}$ the transpose of a matrix $A \in \RR^{n \times m}$.
A smooth function is $C^{\infty}$-smooth and we write $\log$ for the natural logarithm.

\smallskip
{\it Acknowledgements.} I would like to thank Ronen Eldan, Apostolos Giannopoulos, Greg Kuperberg and Alexander Magazinov
for interesting, related discussions. Supported partially by a grant from the European Research Council (ERC) and by the  Mathematical Sciences Research Institute (MSRI)
via
 grant DMS-1440140 of the National Science Foundation (NSF).

\section{Mahler volumes through the Laplace transform}
\label{sec2}

In this section we discuss basic properties of Mahler volumes, the logarithmic Laplace
transform and its Legendre transform.
Let $K \subseteq \RR^n$ be a compact, convex set, and let $p$ be a point belonging to the relative interior of $K$.
The Mahler
volume of $K$ with respect to the point $p$, denoted by
$$ s_p(K) \in (0, \infty), $$
is defined as follows:
There exists a convex body $K_1 \subseteq \RR^{n_1}$ and an affine map $T: \RR^n \rightarrow \RR^{n_1}$ which is a bijection from $K$ to $K_1$  with $T(p) = 0$. We may now
set $s_p(K) := s(K_1)$, and observe that this definition does not depend on the choice of $K_1$ and the map $T$.
Clearly when the origin belongs to the interior of the convex body $K$, we have
$s(K) = s_0(K)$. For a compact, convex set $K \subseteq \RR^n$ we define
$$ \bar{s}(K) = s_{p_K}(K),
$$
where $p_K$ is the Santal\'o point of $K$. It is well-known (see, e.g., Schneider \cite[Section 10.5]{sch}) that
\begin{equation}
 \bar{s}(K) = \inf_{p} s_p(K)
  \label{eq_839}
\end{equation}
  where the infimum runs over all points $p$ in the relative interior of $K$. Moreover, the infimum in (\ref{eq_839})
  is attained at a unique point, the Santal\'o point of $K$, which is in fact the only local minimum of the functional $p \mapsto s_p(K)$.
When  $V \subset \RR^{n+1}$ is a proper, convex cone,
its dual cone is defined via
$$ V^* = \left \{ y \in \RR^{n+1} \, ; \, \forall x \in V,  \langle x, y \rangle \leq 0 \right \}. $$
The dual cone $V^*$ is again a proper, convex cone, and additionally $(V^*)^* = V$.

\begin{lemma} Let $V \subset \RR^{n+1}$ be a proper, convex cone. Let $x_0 \in \i(V)$
	and  $y_0 \in \i(V^*)$ satisfy $\langle x_0, y_0 \rangle = -1$.
Then,
\begin{equation}  s_{x_0}(K) = s_{y_0}(T) = \frac{1}{(n!)^2} \int_V e^{\langle x, y_0 \rangle} dx \cdot \int_{V^*} e^{\langle x_0, y \rangle} dy
\label{eq_1116} \end{equation}
where $K = \{ x \in V \, ; \, \langle x, y_0 \rangle = -1 \}$ and $T = \{ y \in V^* \, ; \, \langle x_0, y  \rangle = - 1 \}$.
\label{lem_1127}
\end{lemma}

\begin{proof} Assume first that there exists a unit vector $e \in S^{n-1}$ with $x_0 = e = - y_0$.
By Fubini's theorem,
	\begin{equation} \int_V e^{-\langle x, e \rangle} dx   = \int_0^{\infty} \int_{\{ z \in V \, ; \, \langle z, e \rangle     = t \}}
	e^{-\langle z, e \rangle} dz dt  = \int_0^{\infty} Vol_n( t K ) \cdot e^{-t } dt  = n! \cdot Vol_n(K).
\label{eq_1102}
	\end{equation}
	By performing a similar computation for $V^*$, we see that
	\begin{equation} \int_V e^{-\langle x, e \rangle} dx \cdot \int_{V^*} e^{\langle e, y \rangle} dy  = (n!)^2 \cdot Vol_n(K) \cdot Vol_n(T) = (n!)^2 \cdot Vol_n(K_1) \cdot Vol_n(T_1),
	\label{eq_1041} \end{equation}
	where we have set $K_1 = K - e$ and $T_1 = T + e$. Note that
	 $K_1$ and $T_1$ are convex bodies in the $n$-dimensional
	linear space $E = e^{\perp} = \{ x \in \RR^n \, ; \, \langle x, e \rangle = 0 \}$.
	Both convex bodies contain the origin in their interior. Let us show that as subsets of the $n$-dimensional Euclidean space $E$,
	the two sets $K_1$ and $T_1$ satisfy
	\begin{equation}
	K_1 = T_1^{\circ}. \label{eq_1050}
	\end{equation}	
	Indeed, a given point $y \in E$ lies in $T_1$
	if and only if $ \langle y - e, x \rangle \leq 0$
	for all points $x \in V$. By homogeneity, it suffices to look only at points $x \in K$, since any $x \in V$ takes the form $x = tz$ for some $z \in K, t \geq 0$.
Thus, a given point $y \in E$ belongs to $T_1$
	if and only if for all $x \in K_1$,
	$$ 0 \geq \langle  y - e, x + e  \rangle = \langle x, y \rangle - \langle e, e \rangle = \langle x,y \rangle - 1. $$
	This proves (\ref{eq_1050}). From (\ref{eq_1041}) and (\ref{eq_1050}) we obtain
	the conclusion of the proposition for the case where $x_0 = -y_0$ is a unit vector.
	
	\smallskip We move on to discuss the general case, which will be reduced to the case analyzed above using linear algebra.
	 Since $\langle x_0, -y_0 \rangle > 0$, there exists a positive-definite, symmetric matrix $P \in \RR^{n \times n}$ with $P x_0 = -y_0$.
	We may decompose $P = S^* S$ for some invertible matrix $S \in \RR^{n \times n}$ and set
	$$ e = S x_0. $$
	Note that $S^* e = S^* S x_0 = P x_0 = -y_0$, and hence $(S^*)^{-1} y_0 = -e$. The vector $e \in \RR^{n+1}$ is a unit vector, as
$$ -1 = \langle x_0, y_0 \rangle = \langle S x_0, (S^*)^{-1} y_0 \rangle = - \langle e, e \rangle. $$
	Denoting $V_1  = S(V)$ we see that $V_1$ is a proper, convex cone with $V_1^* = (S^*)^{-1}  V^*$. 	By changing variables $\tilde{x} = S x$ and $\tilde{y} = (S^*)^{-1} y$ we
obtain
	\begin{equation}
	 \int_V e^{\langle x, y_0 \rangle} dx \cdot \int_{V^*} e^{\langle x_0, y \rangle} dy
	= \int_{V_1} e^{-\langle \tilde{x}, e \rangle} d \tilde{x} \cdot \int_{V_1^*} e^{\langle e, \tilde{y} \rangle} d \tilde{y}.
	\label{eq_1059}
	\end{equation}
	Denote $K_2 = \left \{ \tilde{x} \in V_1 \, ; \, \langle \tilde{x}, -e \rangle = -1 \right \}$
	and $T_2 = 	\left \{ \tilde{y} \in V_1^* \, ; \, \langle e, \tilde{y} \rangle = -1 \right \}$.
	We use (\ref{eq_1059}) and the case of the proposition that was already proven and deduce that
	\begin{equation}
	\frac{1}{(n!)^2} \int_V e^{\langle x, y_0 \rangle} dx \cdot \int_{V^*} e^{\langle x_0, y \rangle} dy = s_e(K_2) = s_{-e}(T_2).
	\label{eq_1107}
	\end{equation}
	However, $S(K) = K_2$ with $S(x_0) = e$ while $(S^*)^{-1}(T) = T_2$ with $(S^*)^{-1}(y_0) = -e$.
	Therefore $ s_e(K_2) = s_{x_0}(K)$ and $s_{-e}(T_2) = s_{y_0}(T)$.
	Thus (\ref{eq_1116}) follows from (\ref{eq_1107}).
\end{proof}

Let $V \subset \RR^{n+1}$ be a proper, convex cone. For $x \in \i(V)$ and $y \in \i(V^*)$ we define
\begin{equation}
 K_y =  \{ z \in V \, ; \, \langle z, y_0 \rangle = -1 \}
\qquad \text{and} \qquad T_x = \{ z \in V^* \, ; \, \langle x, z \rangle = -1 \}. \label{eq_135}
\end{equation}
The notation (\ref{eq_135}) will accompany us throughout this paper. Observe that for any $t > 0$ and $y \in V^*$,
\begin{equation}
\int_V e^{\langle t y, x  \rangle} dx  = \frac{1}{t^{n+1}} \int_V e^{\langle y, x  \rangle} dx. \label{eq_1113}
\end{equation}
By scaling, we obtain the following from Lemma \ref{lem_1127}:

\begin{proposition} Let $V \subset \RR^{n+1}$ be a proper, convex cone. Let $x_0 \in \i(V)$ and $y_0 \in \i(V^*)$ and set $r = -1/\langle x_0, y_0 \rangle$.  Then $K_{y_0}, T_{x_0} \subseteq \RR^{n+1}$ are $n$-dimensional, compact, convex sets with
\begin{equation}  s_{r x_0}(K_{y_0}) = s_{r y_0}(T_{x_0}) =\frac{(-\langle x_0, y_0 \rangle)^{n+1}}{(n!)^2} \int_V e^{\langle x, y_0 \rangle} dx \cdot \int_{V^*} e^{\langle x_0, y \rangle} dy.
\label{eq_1139} \end{equation}
  \label{cor_1137}
\end{proposition}

\begin{remark} {\rm It is possible to interpret the sets $K_{y}$ and $T_{x}$ from  (\ref{eq_135}) as  polar to each other in an appropriate sense: Set $\tilde{K} = K_{y} - x$ and $\tilde{T} = T_{x} - y$.
		Then $\tilde{K}$ is a convex body in the $n$-dimensional subspace $X = y^{\perp}$ while
		$\tilde{T}$ is a convex body in the $n$-dimensional subspace $Y = x^{\perp}$. Moreover,
		\begin{equation}
		\tilde{T} = \left \{ u \in Y \, ; \, \forall v \in \tilde{K}, \ \langle u, v \rangle \leq 1 \right \},
		\qquad \tilde{K} = \left \{ v \in X \, ; \, \forall u \in \tilde{T}, \ \langle u, v\rangle \leq 1 \right \}. \label{eq_823}
		\end{equation}
		Relation (\ref{eq_823}) is the precise duality relation that $\tilde{K}$ and $\tilde{T}$ satisfy.
		In particular, we conclude that the set $K_{y}$ is centrally-symmetric around the point $x$ if and only if the set $T_{x}$ is centrally-symmetric
		around the point $y$. Similarly, $x$ is the barycenter of $K_{y}$ if and only if $y$ is the Santal\'o point of
		$T_{x}$.
	} \label{rem_1208} \end{remark}

The logarithmic Laplace transform of the proper, convex cone  $V \subset \RR^{n+1}$ is the function
\begin{equation}  \Phi_V(y) = \log  \int_V e^{\langle y, x \rangle} dx
\qquad (y \in \RR^{n+1}). \label{eq_411}
\end{equation}
The function $\Phi_V$ is a continuous function from $\RR^{n+1}$ to
$\RR \cup \{ + \infty \}$. It attains a finite value at a point $x \in \RR^{n+1}$ if and only if $x \in \i(V^*)$.
This may be verified, for example, by using formula (\ref{eq_1102}). The function $\Phi_V$ is strictly-convex in $\i(V^*)$, as follows from the Cauchy-Schwartz inequality (see, e.g., \cite{K_quarter}).
It follows from (\ref{eq_1113}) that  $\Phi_V$ has the following homogeneity property:
For any $t > 0$,
\begin{equation}  \Phi_V(t y) = \Phi_V(y) - (n+1) \log t. \label{eq_408}
\end{equation}
Differentiating (\ref{eq_408}) with respect to $t$ we obtain the useful relation
$$ \langle \nabla \Phi_V(y), y \rangle = -(n+1). $$
Proposition \ref{cor_1137} tells us that for any $x \in \i(V)$ and $y \in \i(V^*)$,
\begin{equation}
\Phi_{V^*}(x) + \Phi_V(y) + (n+1) \log (-\langle x, y \rangle)
- 2 \log(n!) = \log s_{r x}(K_y) = \log s_{r y}(T_x), \label{eq_827}
\end{equation}
for $r = -1/\langle x, y \rangle$.
Thus, any local minimum of the Mahler volume corresponds to a local minimum
of the expression on the left-hand side of (\ref{eq_827}).
We would like to compute the first and second variations at a local minimum.
We could have proceeded by  differentiating the expression in (\ref{eq_827}) with respect to $x$ and with respect to $y$, but we find it convenient to eliminate one variable by introducing the Legendre transform.
The Legendre transform
of a convex function $\Phi: \RR^{n+1} \rightarrow \RR \cup \{ + \infty \}$
 is
\begin{equation} \Phi^*(x) = \sup_{y \in \RR^{n+1}, \Phi(y) < \infty} \left[ \langle x, y \rangle - \Phi(y) \right] \qquad
\qquad (x \in \RR^{n+1}). \label{eq_952} \end{equation}
A standard reference for the Legendre transform and convex analysis is Rockafellar \cite{roc}.
The function  $\Phi^*: \RR^{n+1} \rightarrow \RR \cup \{ + \infty \}$ is again convex. If $\Phi$ is finite in a neighborhood
of a point $x \in \RR^{n+1}$ and differentiable at the point $x$ itself, then
\begin{equation}  \Phi^*( \nabla \Phi(x) ) + \Phi(x) = \langle x, \nabla \Phi(x) \rangle. \label{eq_1126}
\end{equation}
In the case where $\Phi = \Phi_V$, the supremum
in (\ref{eq_952}) runs over $y \in \i(V^*)$. Moreover, in this case
we deduce from (\ref{eq_408}) that for any $x \in \i(V)$,
\begin{align} \label{eq_958} \Phi_V^*(x) & = \sup_{s > 0, y \in \i(V^*)} \left[ (n+1) \log s + s \langle x, y \rangle - \Phi_V(y) \right] \\ & =
(n+1) \log \left( \frac{n+1}{e} \right) + \sup_{y \in \i(V^*)} \left[ -(n+1) \cdot \log (-\langle x, y \rangle)  - \Phi_V(y) \right].
\nonumber
\end{align}
Note the difference between the function $\Phi_V^*$, which is the Legendre transform of $\Phi_V$, and the function $\Phi_{V^*}$, which is the logarithmic Laplace transform of the dual cone $V^*$.

\begin{corollary}[``Commuting the Laplace transform with convex duality''] Let $V \subset \RR^{n+1}$ be a proper, convex cone. Then the functions $\Phi_{V^*}$ and $\Phi_V^*$ attain finite values in $\i(V)$, and their difference $J := \Phi_{V^*} - \Phi_V^*$ satisfies
\begin{equation}
 J(x) = \kappa_n + \log \bar{s}(T_x) \label{eq_850} \end{equation}
for $x \in \i(V)$  where $\kappa_n = 2 \log(n!) - (n+1) \log \left(\frac{n+1}{e} \right) $. \label{lem_842}
\end{corollary}

\begin{proof} We already know that $\Phi_{V^*}$ attains finite values in $\i(V)$. Fix $x \in \i(V)$. From (\ref{eq_827}) and (\ref{eq_958}),
	\begin{equation} \Phi_{V^*}(x) - \left[ \Phi_V^*(x) - (n+1) \log \left(\frac{n+1}{e} \right) \right] - 2 \log(n!) = \inf_{y \in \i(V^*)} \log s_{ry} (T_x) \label{eq_1006} \end{equation}
	where $r = -1/\langle x, y \rangle$. When $y$ ranges over the set $\i(V^*)$, the 	point $r y = -y / \langle x, y \rangle$ ranges over the entire relative interior of $T_x$. From (\ref{eq_839}), the right-hand side of (\ref{eq_1006}) equals $\log \bar{s}(T_x)$, and (\ref{eq_850}) follows. The function $\bar{s}(T_x)$ attains finite values in $\i(V)$, as well as the function $\Phi_{V^*}$. By (\ref{eq_850}) the function $\Phi_{V}^*$ is also finite in $\i(V)$.
\end{proof}

Observe that for any homogeneous polynomial $p_k(x)$ of degree $k$
and for any $y \in \i(V^*)$,
\begin{align} \label{poly}
\int_V p_k(x) e^{\langle y, x \rangle} dx
& = \int_0^{\infty} \int_{\{ x \in V \, ; \, \langle x, y  \rangle   = -t |y| \}}
p_k(x) e^{\langle y, x \rangle} dx dt \\ & =  \int_0^{\infty}  (t |y| )^{k+n} e^{-t|y|} dt \cdot \int_{K_y} p_k(x) dx
= \frac{(n+k)!}{|y|} \cdot \int_{K_y} p_k(x) dx. \nonumber
\end{align}
We write $\b(K)$ for the barycenter of a convex body $K$. Recall that $\cov(K_y)$ is the covariance matrix of a random vector that is distributed uniformly in $K_y$.

\begin{lemma}[``Derivatives of $\Phi_V$''] Let $V \subset \RR^{n+1}$ be a proper, convex cone. Then for any $y \in \i(V^*)$,
\begin{equation} \nabla \Phi_V(y) = (n+1) \cdot \b(K_y) \label{eq_1117}
\end{equation}
	and the Hessian matrix is given by
\begin{equation}  \nabla^2 \Phi_V(y) = (n+2) (n+1) \cdot \cov(K_y) + (n+1) \cdot \b(K_y) \b^*(K_y), \label{eq_1118}
\end{equation}
where we view $z \in \RR^{n+1}$ as a column vector while $z^*$ is the corresponding row vector. \label{lem_1122}
\end{lemma}

\begin{proof}  The function $\Phi_V$ is clearly smooth in $\i(V^*)$. By differentiating (\ref{eq_411}) we see that for any $i=1,\ldots,n+1$,
	$$ \frac{\partial \Phi_V(y)}{\partial y_i} = \frac{ \int_V x_i e^{\langle y, x \rangle} dx}{ \int_V e^{\langle y, x \rangle} dx } =  \frac{(n+1)!}{n!} \cdot  \frac{\int_{K_y} x_i dx}{Vol_n(K_y)}, $$
	where we used (\ref{poly}) twice in the last passage. This proves (\ref{eq_1117}). By differentiating (\ref{eq_411}) one more time we see that for $i,j=1,\ldots,n+1$,
\begin{align*}  \nonumber \frac{\partial^2 \Phi_V(y)}{\partial y_i \partial y_j} & = \frac{ \int_V x_i x_j e^{\langle y, x \rangle} dx}{ \int_V e^{\langle y, x \rangle} dx }
\, - \, \frac{ \int_V x_i e^{\langle y, x \rangle} dx}{ \int_V e^{\langle y, x \rangle} dx } \cdot \frac{ \int_V x_j e^{\langle y, x \rangle} dx}{ \int_V e^{\langle y, x \rangle} dx }
\\ & = (n+2) (n+1) \frac{\int_{K_y} x_i x_j dx}{Vol_n(K_y)} - (n+1)^2 \frac{\int_{K_y} x_i dx}{Vol_n(K_y)} \cdot \frac{\int_{K_y} x_j dx}{Vol_n(K_y)},
\end{align*}
 which is equivalent to (\ref{eq_1118}).
\end{proof}

The function $\Phi_V$ is smooth and strictly-convex in $\i(V^*)$, and it equals $+\infty$ outside $\i(V^*)$.
Moreover, $\nabla \Phi_V(y) \in \i(V)$ for any $y \in \i(V^*)$, according to (\ref{eq_1117}).
The function $\Phi_V^*$ is finite in $\i(V)$, and from  the standard theory of the Legendre transform (e.g., Rockafellar \cite[Section 23]{roc}),
for any $x \in \i(V)$ there exists $y \in \i(V^*)$ with $\nabla \Phi_V(y) = x$.

\begin{corollary} For any proper, convex cone $V \subset \RR^{n+1}$, the map $\nabla \Phi_V: \i(V^*) \rightarrow \i(V)$
is a diffeomorphism. \label{cor_diffeo}
\end{corollary}

\begin{proof} We have just explained that the map $\nabla \Phi_V: \i(V^*) \rightarrow \i(V)$ is onto.
Since $\Phi_V$ is strictly-convex, this map is one-to-one. The derivative of this smooth map
is the matrix $\nabla^2 \Phi_V(y)$, which is positive-definite
and hence invertible by Lemma \ref{lem_1122}.
Therefore the map $\nabla \Phi_V: \i(V^*) \rightarrow \i(V)$  is a diffeomorphism.
\end{proof}

It follows from Corollary \ref{cor_diffeo} and formula (\ref{eq_1126}) that the function $\Phi_V^*$ is differentiable in $\i(V)$ and moreover, for any $x \in \i(V)$ and $y \in \i(V^*)$,
\begin{equation}  \nabla \Phi_V^*(x) = y \qquad \Longleftrightarrow \qquad \nabla \Phi_V(y) = x. \label{eq_155}
\end{equation}
In other words, the map $\nabla \Phi_V^*$ is the inverse to the map $\nabla \Phi_V$. Consequently the Hessian matrices
are inverse to each other, that is, for any $x \in \i(V)$,
\begin{equation}  \nabla^2 \Phi_V^*(x) = \left[ \nabla^2 \Phi_V(y) \right]^{-1}, \label{eq_209}
\end{equation}
where $y = \nabla \Phi_V^*(x)$. From
Lemma \ref{lem_1122} and Corollary \ref{cor_diffeo}
we also learn that two hyperplane sections of $V$ coincide if and only if their barycenters coincide.

\section{Projective perturbations and homogeneous cones}

In this section we prove Theorem \ref{thm_1151} and Theorem \ref{prop_356}.
We say that two convex bodies $K, T \subseteq \RR^n$ are {\it projectively equivalent}
or that they are {\it projective images} of one another
if $T$
is affinely-equivalent to a hyperplane section of the cone
\begin{equation} V = \left \{ (t , t x) \in \RR \times \RR^n \, ; \, t \geq 0, x \in K  \right \}. \label{eq_125}
\end{equation}
In other words, for a certain $y \in \i(V^*)$ the set $T$ is affinely equivalent to the convex set $K_y$
from (\ref{eq_135}) that is associated with the cone $V$.

\smallskip
The family of projective images of a convex body $K \subseteq \RR^n$ with a smooth boundary
always contains bodies arbitrarily close to a Euclidean unit ball. In the case where $K \subseteq \RR^n$ is a simplicial
polytope, there are projective images of $K$ that are arbitrarily close to a simplex. A projective
image of a polytope is itself a polytope with the same number of vertices and faces.
A projective image of an ellipsoid is always an ellipsoid, and that of a simplex is
always a simplex. Our next lemma specializes the results of the previous section to
the cone defined in (\ref{eq_125}). We use $x = (t,y) \in \RR \times \RR^n$ as coordinates in $\RR^{n+1}$.

\begin{lemma} Let $K \subseteq \RR^n$ be a convex body with $\b(K) = 0$
and let $V \subset \RR^{n+1}$ be the proper, convex cone defined in (\ref{eq_125}). Denote
$J = \Phi_{V^*} - \Phi_V^*$ and $e = (1,0) \in \RR \times \RR^n \cong \RR^{n+1}$. Then,
\begin{equation}  \nabla J(e) = (n+1) \cdot \left( 0, \b(K^{\circ}) \right) \in \RR \times \RR^n \cong \RR^{n+1}. \label{eq_156}
\end{equation}
In the case where $\nabla J(e) = 0$, the Hessian matrix satisfies
\begin{equation}  \frac{1}{n+1} \cdot \nabla^2 J(e) = (n+2)  \cdot \cov(K^{\circ}) - \frac{1}{n+2} \cdot \cov(K)^{-1}.
\label{eq_201}
\end{equation}
\label{lem_256}
\end{lemma}

A  remark concerning  formula (\ref{eq_201}): The left-hand side is a certain $(n+1) \times (n+1)$ matrix
$A$, while the right-hand side is an $n \times n$ matrix $B$. What we actually mean, is that
$  A = \left( \begin{array}{c|c} 0 & 0 \\ \hline 0 & B \end{array} \right)$. This is consistent with our
coordinates notation, which corresponds to the decomposition $\RR^{n+1} \cong \RR \times \RR^n$.

\begin{proof}[Proof of Lemma \ref{lem_256}] From (\ref{eq_125}) we deduce that
$$ V^* = \left \{ (-t , t x) \in \RR \times \RR^n \, ; \, t \geq 0, x \in K^{\circ}  \right \}. $$
Recalling the notation of (\ref{eq_135}) from the previous section, we see that $T_e = \{ -1 \} \times K^{\circ}$.
By applying Lemma \ref{lem_1122} with the dual cone $V^*$, we obtain
$$ \nabla \Phi_{V^*}(e) = (n+1) \cdot \b(T_e) = (n+1) \cdot \left(-1, \b(K^{\circ}) \right) \in \RR \times \RR^n.
$$
Since $K_{- e} = \{ 1 \} \times K$, when applying Lemma \ref{lem_1122} with the cone $V$ we get
\begin{equation} \nabla \Phi_V(-e) = (n+1) \cdot \b(K_{-e}) = (n+1) \cdot (1, \b(K)) = (n+1) \cdot e. \label{eq_154}
\end{equation}
Recall that $\Phi_V(t y) = \Phi_V(y) - (n+1) \log t$ according to the homogeneity relation (\ref{eq_408}).
By differentiation, we obtain $\nabla \Phi_V(t y) = \nabla \Phi_V(y) / t$, i.e., $\nabla \Phi_V$ is $(-1)$-homogeneous. Therefore, from (\ref{eq_154}),
\begin{equation}  \nabla \Phi_V(-(n+1) \cdot e) = e \qquad \text{and hence} \qquad \nabla \Phi_V^*(e) = -(n+1) e \label{eq_210}
\end{equation}
where we used (\ref{eq_155}) in the last passage. Consequently,
$$ \nabla J(e) = \nabla \Phi_{V^*}(e) - \nabla \Phi_V^*(e) = (n+1) \cdot \left(-1, \b(K^{\circ}) \right) + (n+1) e = (n+1) \cdot (0, \b(K^{\circ})) \in \RR \times \RR^n, $$
proving (\ref{eq_156}). We move on to the proof of (\ref{eq_201}). Since $\nabla J(e) = 0$, then $\b(K^{\circ}) = 0$ and $\b(T_e) = -e$.
According to Lemma \ref{lem_1122},
\begin{equation}
 \nabla^2 \Phi_V(-e) = (n+2) (n+1) \cov(K_{-e}) + (n+1) e e^* = (n+2) (n+1) \cov(K) + (n+1) e e^*
 \label{eq_206} \end{equation}
and
\begin{equation}  \nabla^2 \Phi_{V^*}(e) = (n+2) (n+1) \cov(T_e) + (n+1) e e^* = (n+2) (n+1) \cov(K^{\circ}) + (n+1) e e^*. \label{eq_211}
\end{equation}
Since $\nabla \Phi_V$ is $(-1)$-homogeneous, the Hessian $\nabla^2 \Phi_V$ is $(-2)$-homogeneous, and $\nabla^2 \Phi_V(t y) = \nabla^2 \Phi_V(y) / t^2$. From (\ref{eq_206})
we obtain
$$ \nabla^2 \Phi_V(-(n+1) \cdot e) = \frac{n+2}{n+1} \cdot \cov(K) + \frac{1}{n+1} \cdot e e^*. $$
Thanks to (\ref{eq_209}) and (\ref{eq_210}) we know that
\begin{equation}  \nabla^2 \Phi_V^*(e) = \left[ \nabla^2 \Phi_V (-(n+1) \cdot e) \right]^{-1}  = \frac{n+1}{n+2} \cdot \cov(K)^{-1} + (n+1) \cdot e e^*.
\label{eq_216} \end{equation}
Now (\ref{eq_201}) follows from (\ref{eq_211}), (\ref{eq_216}) and the fact that $J = \Phi_{V^*} - \Phi_V^*$.
\end{proof}

We proceed with a discussion and a proof of Theorem \ref{prop_356}.
Let $V \subset \RR^{n+1}$ be a proper, convex cone.
Denote by $\a(V)$ the group of all invertible, linear transformations $T$ with $T(V) = V$. Clearly $T \in \a(V)$
implies that $T^* \in \a(V^*)$ and vice versa, as for any $x \in \RR^{n+1}$,
$$ \sup_{y \in V} \langle x, y \rangle = \sup_{y \in V} \langle x, T y \rangle = \sup_{y \in V} \langle T^* x, y \rangle. $$
The symmetries of the cone $V$ manifest themselves in the Laplace transform. That is, for any $T \in \a(V)$ and $y \in \i(V^*)$,
\begin{equation}  \Phi_V(  T^* y)  = \log  \int_V e^{\langle y, T x \rangle} dx = \log  \int_V e^{\langle y,  x \rangle} dx - \log |\det T|
= \Phi_V(y) - \log |\det T|.
\label{eq_350} \end{equation}
Consequently, for any $x \in \i(V)$ and $T \in \a(V)$,
\begin{equation}  \Phi_V^*(T x) = \sup_{y \in \i(V^*)} \left[ \langle x, T^* y \rangle - \Phi_V( y)  \right] = \Phi_V^*(x) - \log |\det T|. \label{eq_351}
\end{equation}

\begin{proof}[Proof of Theorem \ref{prop_356}] We may assume that the barycenter of $K$ lies at the origin
and define $V$ as in (\ref{eq_125}). Then $V$ is a convex, homogeneous cone. From (\ref{eq_350}) and (\ref{eq_351}) we know that
for any $x \in \i(V)$ and $T \in \a(V)$,
$$ J(T x) = \Phi_{V^*}(T x) - \Phi_{V}^*( T x) = \left( \Phi_{V^*}(x) - \log | \det T| \right) - \left( \Phi_{V}^*( x) - \log |\det T| \right) = J(x). $$
However, for any $x,y \in \i(V)$ there is $T \in \a(V)$ with $T x = y$. Consequently $J: \i(V) \rightarrow \RR$ is a constant function.
In particular, the gradient and the Hessian matrix of $J$ vanish. We may now apply the computations of Lemma \ref{lem_256}.
First, we deduce that $\b(K^{\circ}) = 0$, and hence the Santal\'o point of $K$
coincides with its barycenter. Second, we obtain  $$ (n+2)  \cdot \cov(K^{\circ}) = \frac{1}{n+2} \cdot \cov(K)^{-1}. $$
In particular, $ L_K^{2n} \cdot L_{K^{\circ}}^{2n} \cdot s(K)^2 = \det \cov(K) \cdot \det \cov(K^{\circ}) = (n+2)^{-(2n)}$ by (\ref{eq_212}).
\end{proof}

\begin{remark}{\rm  A convex body $K \subseteq \RR^n$
is affinely equivalent to a hyperplane section of a homogeneous cone if and only if every projective image of $K$ is affinely equivalent to $K$. This  follows from the fact that two hyperplane sections of a proper, convex cone $V \subset \RR^{n+1}$ coincide if and only if their barycenters coincide.
A corollary is that the dual to a homogeneous cone is homogeneous in itself. These standard facts are not used in this paper. }
\end{remark}
Theorem \ref{thm_1151} will be deduced from the next proposition:

\begin{proposition} Let $T \subseteq \RR^n$ be a convex body which is a local minimizer of the Mahler volume
$s(T)$	in the class of the projective images of $T$.  Then,
	\begin{equation}  (n+2)^2 \cdot \cov(T^{\circ}) \geq \cov(T)^{-1}. \label{eq_208}
	\end{equation}
Moreover, if $T$ is a local maximizer of the Mahler volume in the class of projective images of $T$ with barycenter at the origin, then $(n+2)^2 \cdot \cov(T^{\circ}) \leq \cov(T)^{-1}$.
	\label{prop_1021}
\end{proposition}

\begin{proof} A translation of $T$ is a particular case of a projective image of $T$.
Assume that $T$ is a local minimizer. From (\ref{eq_839}) we learn that  the Santal\'o point of $T$
lies at the origin.
Denote $K = T^{\circ}$, so $\b(K) = 0$. Consider the proper, convex cone $V \subset \RR \times \RR^n$ defined in (\ref{eq_125}). For $e = (1,0) \in \RR \times \RR^n$ we
have
$$ T_e = \{ -1 \} \times T. $$
Therefore $T_x$ is a projective image of $T$ for any $x \in \i(V)$. When $x$ is close to $e$, the Santal\'o point of $T_x$ is close to the Santal\'o point of $T_e$. By our local minimum assumption, for any $x$ in some neighborhood of $e$, we have
$$ \bar{s}(T_x) \geq \bar{s}(T_e) = \bar{s}(T) =  s(T). $$
Recall from Corollary \ref{lem_842} that $ J(x) = \Phi_{V^*}(x) - \Phi_V^*(x) = \kappa_n + \log \bar{s}(T_x) $. We thus conclude
that $e$ is a local minimum of the function $J: \i(V) \rightarrow \RR$. Thus $\nabla J(e) = 0$ and $\nabla^2 J(e) \geq 0$.
From Lemma \ref{lem_256} we obtain
$$
(n+2)  \cdot \cov(T) - \frac{1}{n+2} \cdot \cov(T^{\circ})^{-1} = (n+2)  \cdot \cov(K^{\circ}) - \frac{1}{n+2} \cdot \cov(K)^{-1} \geq 0. $$
This proves (\ref{eq_208}). Similarly, if $T$ is a local maximizer of $\bar{s}(T_x)$, then $\nabla J(e) = 0$
and $\nabla^2 J(e) \leq 0$, and from Lemma \ref{lem_256} we obtain
$ (n+2)^2 \cdot \cov(T^{\circ}) \leq \cov(T)^{-1}$. \end{proof}

\begin{proof}[Proof of Theorem \ref{thm_206}]
Assume that $K \subseteq \RR^n$ is a local minimizer
of the Mahler volume in the class of convex bodies containing the origin in their interior. In particular, $K$
is a local minimizer in the class of projective images of $K$, and by Proposition \ref{prop_1021},
\begin{equation} \cov(K^{\circ}) - (n+2)^{-2} \cdot \cov(K)^{-1} \geq 0. \label{eq_305}
\end{equation}
In order to prove (\ref{eq_947}), we use the fact that $\det(A) \geq \det(B)$ whenever $A \geq B \geq 0$. Thus,
by (\ref{eq_212}) and (\ref{eq_305}),
$$ L_K^{2n} \cdot L_{K^{\circ}}^{2n} \cdot s(K)^2 = \det \cov(K) \cdot \det \cov(K^{\circ}) \geq \frac{1}{(n+2)^{2n}}, $$
proving (\ref{eq_947}). Note that if $K$ is a global minimizer, then $s(K) \leq s(\Delta^n)$ and from (\ref{eq_947}),
$$ L_K \cdot L_{K^{\circ}} \geq L_{\Delta^n}^2 $$
where we used the fact that $L_{\Delta^n}^2 \cdot s(\Delta^n)^{1/n} = 1/(n+2)$. Hence $L_K \geq L_{\Delta^n}$ or $L_{K^{\circ}} \geq L_{\Delta^n}$
in the case of a global maximizer.
\end{proof}

\begin{remark}{\rm Suppose that $K \subseteq \RR^n$ is a convex body which is not an ellipsoid, whose boundary is smooth
with Gauss curvature that is always positive (i.e., the boundary is strongly-convex).
Assume that the barycenter of $K$ lies at the origin, set $T = K^{\circ}$, and let  $V$ be defined
as in (\ref{eq_125}). Then by Corollary \ref{lem_842} with $x = e = (1,0) \in \RR \times \RR^n$,
$$ J(e) = \kappa_n + \log \bar{s}(T) < \kappa_n + \log s(B^n), $$
where in the last passage we used
the equality case in the
Santal\'o inequality (see Meyer and Pajor \cite{MP}).
 We claim that for $J = \Phi_{V^*} - \Phi_V^*$,
\begin{equation} \lim_{x \rightarrow \infty} J(x) = \kappa_n + \log s(B^n) > J(e).
\label{eq_259} \end{equation}
Indeed, the behavior of the functional $J: \i(V) \rightarrow \RR$
at infinity is simple to understand, as the corresponding hyperplane sections $T_x$ are close to ellipsoids,
and hence $\bar{s}(T_x)$ is  close to $s(B^n)$.
From (\ref{eq_259}) we see that the infimum of $J$ is necessarily attained at some point $x \in \i(V)$.
From Lemma \ref{lem_256}
we thus obtain that whenever $T \subseteq \RR^n$
is a convex body with a smooth and strongly-convex boundary, it has a projective image $\tilde{T}$
whose barycenter and Santal\'o point are at the origin and
$$ (n+2)^2 \cdot \cov(\tilde{T}^{\circ}) \geq \cov(\tilde{T})^{-1}. $$
} \label{rem_1159} \end{remark}

\section{Examples}

Let us begin this section by inspecting the simplest example, the case of the orthant
$$ V = \RR^{n+1}_+ = \left \{ x = (x_0,\ldots,x_n) \in \RR^{n+1} \, ; \, \forall i, \ x_i \geq 0 \right \}. $$
The proper, convex cone $V$ is homogeneous, with an automorphism group consisting of all diagonal transformations
with positive entries on the diagonal. Moreover, in this case $V^* = - V$, and a homogeneous cone with this property is called
a {\it symmetric cone}.
For any $y \in \i(V^*)$, the set $K_y$ is an $n$-dimensional simplex. The logarithmic Laplace transform
is given by
$$ \Phi_V(y) = - \sum_{i=0}^n \log |y_i| \qquad \qquad \text{for} \ y \in \i(V^*). $$
Since $\sup_{s > 0} [-st + \log(s)] = -1 - \log(t)$ for any positive $t$, we have
$$ \Phi_V^*(x) = -(n+1) - \sum_{i=0}^n \log |x_i| \qquad \qquad \text{for} \ x \in \i(V). $$
It follows that  $\Phi_{V^*}(x) - \Phi_V^*(x) = n+1$ for any $x \in \i(V)$.
From Corollary \ref{lem_842} and this example we conclude the following:

\begin{corollary} The Mahler conjecture (\ref{eq_1001}) is equivalent to the
assertion that for any proper, convex
cone $V \subset \RR^{n+1}$ and for any point $x \in \i(V)$ we have $\Phi_{V^*}(x) - \Phi_V^*(x) \geq n+1$.
\label{cor_502}
\end{corollary}

The second  example we consider is where $$ V = \left \{ (x_0,\ldots, x_n) \in \RR^{n+1} \, ; \, x_0 \geq \sqrt{ \sum_{i=1}^n x_i^2} \right \} $$
is the Lorentz cone. Here again $V^* = -V$. For any $y \in \i(V^*)$, the set $K_y$ is an ellipsoid.
Denoting $Q(x) = x_0^2 - \sum_{i=1}^n x_i^2$ for $x = (x_0,\ldots,x_n) \in \RR^{n+1}$, we have
$$  \Phi_V( y) =  - \frac{n+1}{2} \log Q(y) + C_n,   \qquad \qquad \text{for} \ y \in \i(V^*), $$
as  dictated by the symmetries of the problem, where $C_n = \log ( \pi^{n/2} \cdot \Gamma(n+1) / \Gamma(1 + n/2) )$. Moreover, here
\begin{equation} \Phi_{V^*} - \Phi_V^* \equiv 2 C_n - (n+1) \log \frac{n+1}{e}, \label{eq_338}
\end{equation}
which by Santal\'o's inequality is the maximal possible value of $\Phi_{V^*} - \Phi_V^*$ for any proper, convex cone $V \subset \RR^{n+1}$.
The right-hand side of (\ref{eq_338}) is asymptotically $(\log(2\pi) + o(1)) \cdot n$, according to Stirling's approximation.
 The third example we consider is the cone of positive semi-definite matrices
 \begin{equation} V = \left \{ A \in \RR^{n \times n} \, ; \, A^* = A, A \geq 0 \right \}. \label{eq_355} \end{equation}
This is a proper, convex cone in the linear space $X_n \subseteq \RR^{n \times n}$ of all real, symmetric matrices. We endow $X_n$ with the scalar product
$$ \langle A, B \rangle = \tr[A B] = \sum_{i=1}^n A_{ii} B_{ii} + 2 \sum_{i < j} A_{ij} B_{ij} $$
where $A = (A_{ij})_{i,j=1,\ldots,n}$ and $B = (B_{ij})_{i,j=1,\ldots,n}$. With this scalar product, we have $V^* = -V$.
The volume form in $X_n$
which is induced by this scalar product is the volume form $d A := 2^{n(n-1)/4} \prod_{i \leq j} d A_{ij}$, up to a sign which corresponds to
a choice of orientation in $X_n$.
The logarithmic Laplace transform is given by
\begin{equation}  \Phi_V(-A) = \log \int_V e^{-\tr[A B]} d B =   -\frac{n+1}{2} \cdot \log \det A  + C_n \label{eq_354} \end{equation}
for some constant $C_n$. Indeed, for any map $T \in \RR^{n \times n}$ with $\det(T) = 1$ we know
that $\Phi_V(T^* A T) = \Phi_V(A)$. This shows that $\Phi_V(A)$ depends only on the determinant of $A$.
The homogeneity property  $\Phi_V(t A) = -\log(t) \cdot n(n+1)/2 + \Phi_V(A)$ implies formula (\ref{eq_354}).
The computation of $C_n$ by induction on $n$ is well-known and it is explained, e.g., in \cite[Section 5.7]{Neretin}.
For completeness, and since our notation is a bit different, the following lemma includes the standard computation:

\begin{lemma} In the case where $V = V_n$ is given by (\ref{eq_355}), the constant $C_n$ from (\ref{eq_354}) satisfies
\begin{equation} C_n = \log \int_{V_n} e^{-\tr[A]} dA = \frac{n(n-1)}{4} \log (2\pi) + \sum_{k=1}^{n}  \log \Gamma \left( \frac{k+1}{2} \right).
\label{eq_456}
\end{equation}
\end{lemma}

\begin{proof} For $A \in X_n$ let us write
$$  A = \left( \begin{array}{c|c} B & u \\ \hline u^* & s \end{array} \right) $$
where $B \in X_{n-1}$ is a symmetric matrix, $u \in \RR^{n-1}$ and $s \in \RR$.
Then $d A = \pm 2^{(n-1)/2} d B \wedge du \wedge ds$, where we recall that $d B := 2^{(n-1)(n-2)/4} \prod_{i \leq j} d B_{ij}$
while $d u = \prod_{i} d u_i$. By setting $v = u/s, $ and $D = B - s v v^*$ we obtain
$$ A = \left( \begin{array}{c|c} B & u \\ \hline u^* & s \end{array} \right)
= \left( \begin{array}{c|c} D + s v v^* & s v \\ \hline s v^* & s \end{array} \right)
=
\left( \begin{array}{c|c} 1 & v \\ \hline 0 & 1 \end{array} \right) \cdot \left( \begin{array}{c|c} D & 0 \\ \hline 0 & s \end{array} \right) \cdot
\left( \begin{array}{c|c} 1 & 0 \\ \hline v^* & 1 \end{array} \right). $$
Note that
the map from $A \in \i(V_n)$ to $(D, v, s) \in \i(V_{n-1}) \times \RR^{n-1} \times (0, \infty)$ is a diffeomorphism. Moreover,  $d u = s^{n-1} dv + \alpha \wedge ds$ for some $\alpha$ and hence $d B \wedge du \wedge ds = s^{n-1} d D \wedge d v \wedge ds$.
 Consequently,
\begin{align*}
e^{C_n} & = \int_{V_n} e^{-\tr[A]} dA = 2^{\frac{n-1}{2}} \int_{V_{n-1} \times \RR^{n-1} \times (0, \infty)} e^{-\tr[D] - s |v|^2 - s} s^{n-1} d D \wedge d v \wedge ds
\\ & = e^{C_{n-1}} \cdot 2^{\frac{n-1}{2}} \int_0^{\infty} s^{n-1} e^{-s} \left( \int_{\RR^{n-1}} e^{-s |v|^2} dv    \right) ds = e^{C_{n-1}} 2^{\frac{n-1}{2}} \int_0^{\infty} s^{n-1} e^{-s} \left( \frac{\pi}{s} \right)^{(n-1)/2} ds.
\end{align*}
It follows that
$$ C_n = C_{n-1} + \frac{n-1}{2} \log (2 \pi) + \log \Gamma \left( \frac{n+1}{2} \right). $$
Since $C_1 = 0$, formula (\ref{eq_456}) follows by a simple induction.
\end{proof}
It follows that  in the case where $V$ is given by (\ref{eq_355}),
$$ \Phi_{V^*} - \Phi_V^* \equiv 2 C_n - \frac{n(n+1)}{2} \log \frac{n+1}{2e} = \frac{n (n+1)}{2} \cdot \left[ \log(2\pi) - 1/2 + o(1) \right] $$
where the asymptotics as $n \rightarrow \infty$ follows from Stirling's formula. Since $\log(2 \pi) - 1/2 > 1$,
in view of Corollary \ref{cor_502} we have verified
the Mahler conjecture for hyperplane sections
of the cone of positive-definite, symmetric
matrices in a sufficiently high dimension.
Moreover, thanks to Theorem \ref{prop_356}, we see that in high dimensions, the isotropic constant of such hyperplane sections
is smaller than that of the simplex in the corresponding dimension.
We proceed with the case of complex-valued matrices
and the cone
 \begin{equation} V = \left \{ A \in \CC^{n \times n} \, ; \, A^* = A, A \geq 0 \right \}. \label{eq_508} \end{equation}
Now we write $X_n$ for the space of Hermitian $n \times n$ matrices, equipped with the scalar product
 $ \langle A, B \rangle = \tr[A^* B]$. This space has real dimension $n^2$, and we have $V^* = -V$.
 The induced volume form is $d A := 2^{n(n-1)/2} \prod_{i \leq j} d A_{ij}$,
 where $A_{ii} \in \RR$ and $A_{ij} \in \CC$ for $i \neq j$.
 The logarithmic Laplace transform is given by
$$  \Phi_V(-A) = \log \int_V e^{-\tr[A B]} d B =   -n \cdot \log \det A  + C_n $$
where
$$ C_n = \frac{n(n-1)}{2} \log (2 \pi) + \sum_{k=1}^{n-1} \log(k!). $$
Here,
$$ \Phi_{V^*} - \Phi_V^* \equiv 2 C_n - n^2 \log \left( \frac{n}{e} \right) = n^2 \cdot \left[  \log(2\pi) - 1/2 + o(1) \right]. $$
Once again we see  the numerical constant $\log(2 \pi) - 1/2$ which appeared in the case of real, symmetric matrices.
The same numerical constant also appears  in
the  quaternionic case.

\smallskip A natural operation on convex cones is that of Cartesian products. If $V_1 \subseteq \RR^{n_1 + 1}$ and $V_2 \subseteq \RR^{n_2 + 1}$
are proper, convex cones, then so is the Cartesian product $V_1 \times V_2 \subseteq \RR^{n_1 + 1} \times \RR^{n_2 + 1}$ whose dual is $V_1^* \times V_2^*$.
Moreover,
$$ \Phi_{V_1 \times V_2}(x,y) = \Phi_{V_1}(x) + \Phi_{V_2}(y) \qquad \qquad (x \in \i(V_1^*), y \in \i(V_2^*)) $$
and similarly $\Phi_{V_1 \times V_2}^*(x,y) = \Phi^*_{V_1}(x) + \Phi^*_{V_2}(y)$. Write $J_V(x) = \Phi_{V^*}(x) - \Phi_V^*(x) $
and let
$$ J_{n+1} := \inf_{V \subset \RR^{n+1}} \inf_{x \in \i(V)} J_V(x) $$
where the first infimum runs over all proper, convex cones
$V \subset \RR^{n+1}$. Since $J_{V_1 \times V_2}(x,y) = J_{V_1}(x) + J_{V_2}(y)$,
we see that $J_{n+m} \leq J_n + J_m$. The subadditivity property of $J_n$ implies that
$$ \lim_{n \rightarrow \infty} \frac{J_n}{n} = \inf_{n \rightarrow \infty} \frac{J_n}{n} $$
thanks to the Fekete lemma. Thus, in view of Corollary \ref{cor_502}, the Mahler conjecture
would follow from an asymptotic estimate of the form $J_n \geq (1 + o(1)) \cdot n$. We move on to discuss the case where  for $i=1,2$, the cone $V_i \subseteq \RR^{n_i}$ takes the form
\begin{equation}
 V_i = \left \{ (t, tx) \, ; \, t \geq 0, x \in K_i \right \} \subseteq \RR \times \RR^{n_i} \label{eq_1452}
 \end{equation}
for some convex body $K_i \subseteq \RR^{n_i}$. Consider the hyperplane section of the cone
$V_1 \times V_2$ that consists of all $4$-tuples $(t,x, s,y)$ with $t + s = 1$. This hyperplane section is
$$ \{ (t, tx, 1-t, (1-t) y) \, ; \, x \in K_1, y \in K_2, 0 \leq t \leq 1 \}, $$
which is affinely equivalent to the {\it geometric join} of  $K_1$ and $K_2$, defined via
\begin{equation} K_1 \diamondsuit K_2 := \sqrt{2} \cdot \{ (t-1/2, tx, 1/2-t, (1-t) y) \, ; \, x \in K_1, y \in K_2, 0 \leq t \leq 1 \}. \label{eq_1137}
\end{equation}
The geometric join of $K_1 \subseteq \RR^{n_1}$ and $K_2 \subseteq \RR^{n_1}$ is an $(n_1 + n_2+1)$-dimensional compact, convex set
with a non-empty interior relative to the ambient linear subspace
$$ H_{n_1, n_2} = \{ (t, x, -t, y) \, ; \, t \in \RR, x \in \RR^{n_1}, y \in \RR^{n_2} \} \subseteq \RR^{n_1 + n_2 + 2}.
$$ Our definition of a geometric join is slightly
different from the perhaps more standard notation
 in \cite{Ku2}, yet the two definitions are affinely equivalent.
 The geometric join of $\Delta^{n_1}$ and $\Delta^{n_2}$ is an $(n_1 + n_2 + 1)$-dimensional simplex.
 Geometric joins fit well with duality:

\begin{proposition} Let $K_1 \subseteq \RR^{n_1}$ and $K_2 \subseteq  \RR^{n_2}$ be convex bodies containing the origin in their interior.
Then,
\begin{equation}
 (K_1 \diamondsuit K_2)^{\circ} = \pi( K_1^{\circ} \diamondsuit K_2^{\circ} ),
\label{eq_941}
\end{equation}
where we view $K_1 \diamondsuit K_2$ and $K_1^{\circ} \diamondsuit K_2^{\circ}$ as convex bodies
in the subspace $H_{n_1, n_2} \subseteq \RR^{n_1 + n_2 + 2}$ which is equipped with the induced scalar product from $\RR^{n_1 + n_2 + 2}$, and where
$$ \pi(t, x, -t, y) = (-t, x, t, y) \qquad \text{for} \
(t,x,-t,y) \in H_{n_1, n_2}. $$ \label{prop_203}
\end{proposition}

\begin{proof} This follows from the fact that $(V_1 \times V_2)^* = V_1^* \times V_2^*$ where $V_i$ is given by
	(\ref{eq_1452}). Alternatively, we may argue directly as follows.
	A point $\sqrt{2} \cdot (1/2-t,x,t-1/2,y) \in H_{n_1,n_2}$ belongs to $(K_1 \diamondsuit K_2)^{\circ}$
if and only if for all $s \in [0,1], x' \in K_1$ and $y' \in K_2$,
$$ -2  (t-1/2) (s-1/2) + s \langle x, x' \rangle + (1-s) \langle y,  y' \rangle \leq \frac{1}{2}. $$
This happens if and only if for all $s \in [0,1]$,
\begin{equation}
 -2  (t -1/2) (s-1/2) + s \| x \|_{K_1^{\circ}} + (1-s) \| y \|_{K_2^{\circ}} \leq \frac{1}{2} \label{eq_1505}
 \end{equation}
where $\| x \|_{K_i^{\circ}} = \sup_{x' \in K_i} \langle x,  x' \rangle$. Since the left-hand side of (\ref{eq_1505}) is an affine function of $s$, it suffices to look at the two values $s=0,1$. Hence the
point $\sqrt{2} \cdot (1/2-t,x,t-1/2,y) $ belongs
to $(K_1 \diamondsuit K_2)^{\circ}$ if and only if
$$ (1/2 - t) + \| x \|_{K_1^{\circ}} \leq 1/2
\qquad \text{and} \qquad (t - 1/2) + \| y \|_{K_2^{\circ}} \leq 1/2, $$
i.e., if and only if $x \in t K_1^{\circ}$ and $y \in (1 - t) K_2^{\circ}$. This completes the proof.
\end{proof}

The geometric join may be viewed as a variant of the Cartesian product. For example, it may be verified using Corollary \ref{lem_842}
and the connection between geometric joins and Cartesian products of cones,  that
$$  \bar{s}(K_1 \diamondsuit K_2) =  C_{n_1,n_2}  \cdot \bar{s}(K_1) \cdot \bar{s}(K_2) $$
for any convex bodies $K_1 \subseteq \RR^{n_1}$ and $K_2 \subseteq \RR^{n_2}$, where
$$ C_{n_1,n_2} =  \left( \frac{n_1! \cdot n_2!}{(n_1 + n_2 + 1)!} \right)^2 \cdot \frac{(n_1 + n_2 + 2)^{n_1 + n_2 + 2}}{(n_1 + 1)^{n_1 + 1} (n_2 + 1)^{n_2 + 1}}. $$
In comparison, for Cartesian products we know that for any convex bodies $K_1 \subseteq \RR^{n_1}$ and $K_2 \subseteq \RR^{n_2}$,
$$ \bar{s}(K_1 \times K_2) = \tilde{C}_{n_1, n_2} \cdot \bar{s}(K_1) \cdot \bar{s}(K_2) $$
where $\tilde{C}_{n_1, n_2} = n_1! n_2! / (n_1 + n_2)!$. For more relations between
the geometric join and various inequalities, see Rogers and Shephard \cite{RS}.

\section{Covariance of a body and its polar}

In this section we prove Proposition \ref{prop_1157}. Let $K \subseteq \RR^n$ be a centrally-symmetric convex body.
Then $\b(K) = \b(K^{\circ}) = 0$.
From Corollary \ref{lem_842} and Lemma \ref{lem_256} we learn the following: The
 functional
$$ T \mapsto \bar{s}(T), $$ restricted to the class of projective images $T$ of the body $K$,
has a stationary point at $K$. If this stationary point were in fact a {\it local maximum} then by Proposition \ref{prop_1021} we would have
\begin{equation} \phi(K) = \Tr[ \cov(K^{\circ}) \cdot \cov(K) ] \leq n / (n+2)^2. \label{eq_452}
\end{equation}
Inequality (\ref{eq_452}) is equivalent to Conjecture 5.1 from \cite{Ku}.
This local maximum property indeed holds in the case where $K$ is the unit ball
of $\ell_p^n$ for $1 \leq p \leq \infty$, see Alonso-Guti\'errez \cite{Alonso}.
However, Proposition \ref{prop_1157} above implies that this local maximum property fails in general.
The remainder of this section is concerned with the proof of
Proposition \ref{prop_1157}.
Set
$$ K_0 = B_1^{n-1} = \left \{ x \in \RR^{n-1} \, ; \, \sum_{i=1}^{n-1} |x_i| \leq 1 \right \}. $$
Define $K_1 = K_0 \cap (\sqrt{3/n}) B_2^{n-1}$ where $B_2^{n-1} = \{ x \in \RR^{n-1} \, ; \sum_i |x_i|^2 \leq 1 \}$, and
\begin{equation} K = \left \{ (t,x) \in \RR \times \RR^{n-1} \, ; \, |t| \leq 1, \ x \in (1 - |t|) K_0 + |t| K_1 \right \}.
\label{eq_1228}
\end{equation}
We claim that
\begin{equation}
Vol_{n-1}(K_1)  \geq \frac{1}{3} \cdot Vol_{n-1}(K_0).
\label{eq_1221} \end{equation}
Indeed, a direct computation shows that $\int_{K_0} x_i^2 dx = 2 Vol_{n-1}(K_0) / [n (n+1)]$ for all $i$.
Therefore the average of the function $x \mapsto |x|^2$ on $K_0$ is at most $2/n$. Now (\ref{eq_1221}) follows by the Markov-Chebychev inequality. From (\ref{eq_1228}) and (\ref{eq_1221})
we conclude that
$$ \forall x_1 \in (-1,1), \qquad  \frac{1}{3} \leq \frac{Vol_{n-1} \left( \left \{ x \in \RR^{n-1} \, ; \, (x_1, x) \in K \right \} \right)}
{Vol_{n-1}(K_0)} \leq 1,
$$
where again we use the coordinates $(x_1, x) \in \RR \times \RR^{n-1} \cong \RR^n$. Consequently
$Vol_n(K) \leq 2 Vol_{n-1}(K_0)$ and
\begin{equation} \int_K x_1^2 \frac{dx}{Vol_n(K)}
= \int_{-1}^1 x_1^2 \cdot \frac{Vol_{n-1} \left( \left \{ x \in \RR^{n-1} \, ; \, (x_1, x) \in K \right \} \right)}
{Vol_{n}(K)} dx_1 \geq \int_{-1}^1 \frac{x_1^2}{6} dx_1 = \frac{1}{9}.
\label{eq_1232} \end{equation}
We now move on to discuss the integral of $x_1^2$ over $K^{\circ}$. We require the following:

\begin{lemma} Let $X_1,\ldots,X_{n-1}$ be independent random variables, distributed uniformly in the interval $[-1,1]$.
Then with a probability of at least $1/6$, there exists a decomposition of $X = (X_1,\ldots,X_{n-1})$ as
$$ X = (Y + Z) / 2 $$
with $Y \in (8/9) B_{\infty}^{n-1}$ and $Z \in \sqrt{3 n / 10} \cdot B_2^{n-1}$. Here, $B_{\infty}^{n-1} = [-1,1]^{n-1}$.
\label{lem_214}
\end{lemma}

\begin{proof} We set
$$ Y_i = \left \{ \begin{array}{cl} 8/9 & \textrm{if } X_i > 4/9 \\ 2 X_i & \textrm{if} \ -4/9 \leq X_i \leq 4/9 \\ -8/9 & \textrm{if } X_i < -4/9 \end{array} \right. $$
and $Z_i = 2 X_i - Y_i$. Then for any $i$,
$$ \EE Z_i^2 = \int_{4/9}^1 (2t - 8/9)^2 dt = \int_0^{5/9} (2s)^2 ds = \frac{4}{3} \cdot \left( \frac{5}{9} \right)^3 < \frac{1}{4}. $$
Hence $\EE |Z|^2 < n / 4$, and consequently $\PP( |Z| \geq \sqrt{3 n / 10}) \leq 5/6$ by the Markov-Chebyshev inequality.
\end{proof}

It follows from  (\ref{eq_1228}) that
$$ K^{\circ} = \left \{ (t, x) \in \RR \times \RR^{n-1} \, ; \, |t| \leq 1, x \in K_0^{\circ} \ \textrm{and} \ x \in (1 - |t|) \cdot K_1^{\circ} \right \}. $$
Note that  $K_0^{\circ} = B_{\infty}^{n-1}$ while
$K_1^{\circ}$ is the convex hull of $B_{\infty}^{n-1}$ with $\sqrt{n / 3} \cdot B_2^{n-1}$. Lemma \ref{lem_214} implies that
for any $|t| \leq 1/20$,
\begin{align} \label{eq_231}
Vol_{n-1} & \left( \left \{ x \in \RR^{n-1} \, ; \, (t, x) \in K^{\circ} \right \} \right) \geq Vol_{n-1} \left(
\left \{ x \in B_{\infty}^{n-1} \, ; \, x \in \frac{19}{20}  K_1^{\circ} \right \} \right) \\ &
\geq Vol_{n-1} \left(
\left \{ x \in B_{\infty}^{n-1} \, ; \, x \in \frac{19}{20} \cdot \frac{B_{\infty}^{n-1} +  \sqrt{n / 3} \cdot B_2^{n-1}}{2}  \right \} \right) \geq \frac{1}{6} \cdot
Vol_{n-1} \left( B_{\infty}^{n-1} \right). \nonumber
\end{align}
Write $\alpha(t) = Vol_{n-1} \left( \left \{ x \in \RR^{n-1} \, ; \, (t, x) \in K^{\circ} \right \} \right)$.
Then $\alpha$ is supported in $[-1,1]$, and its maximum is attained at $t = 0$ by the Brunn-Minkowski inequality.
From (\ref{eq_231}) we learn that $\alpha(t) \geq \alpha(0) / 6$ for $|t| \leq 1/20$. Therefore,
\begin{equation} \int_{K^{\circ}} x_1^2 \frac{dx}{Vol_n(K^{\circ})} = \frac{\int_{-1}^1 s^2 \alpha(s) ds}{\int_{-1}^1 \alpha(s) ds} \geq \frac{\int_{-1/20}^{1/20} s^2 \cdot (\alpha(0) / 6) ds   }{2 \alpha(0)} \geq 10^{-6}. \label{eq_234} \end{equation}
Glancing at (\ref{eq_1228}) we see that the compact set $K \subseteq \RR^n$ is convex and unconditional.
The conclusion (\ref{eq_425}) of Proposition \ref{prop_1157} thus follows from (\ref{eq_1232}) and (\ref{eq_234}).
Since $K$ and $K^{\circ}$ are unconditional, their covariance matrices are positive-definite and diagonal. Hence,
with $e_1 = (1,0,\ldots,0) \in \RR^n$,
\begin{align*} \phi(K) & = \tr[ \cov(K^{\circ}) \cdot \cov(K) ] \geq \langle \cov(K^{\circ}) e_1, e_1 \rangle \cdot \langle
\cov(K) e_1, e_1 \rangle  \\ & = \int_K x_1^2 \frac{dx}{Vol_n(K)} \cdot \int_{K^{\circ}} x_1^2 \frac{dx}{Vol_n(K^{\circ})} \geq c.
\end{align*}
This completes the proof of Proposition \ref{prop_1157}.

\section{The floating body of a cone and self-convolution}

In this section we describe
various relations between the floating body of a convex cone, its Laplace transform
and its self-convolution.
 Given a convex set $A \subseteq \RR^n$ and a parameter $\delta > 0$, Sch\"utt and Werner \cite{SW} define the {\it floating body} $A_\delta$ as the intersection
of all closed half-spaces $H \subseteq \RR^n$ for which
$$ Vol_n(A \cap H) \geq \delta. $$
The floating body $A_{\delta} \subseteq A$ is closed and convex. When $V \subset \RR^{n+1}$ is a proper, convex cone, by homogeneity we have
$$ V_{\delta} = \delta^{1/(n+1)} \cdot V_1 \qquad \text{for all} \ \delta > 0.
$$
Clearly $V_{\delta} \subseteq \i(V)$. Recall the logarithmic Laplace transform $\Phi_V$ and its Legendre transform $\Phi_V^*$.

\begin{proposition} For any proper, convex cone $V \subset \RR^{n+1}$ and $\delta > 0$, we have $$ V_{\delta} = \left \{ x \in \i(V)
 \, ; \, \Phi_V^*(x) \leq \kappa_n - \log \delta \right \} $$ where  $\kappa_n = \log \left[ \left(\frac{n+1}{e} \right)^{n+1} / (n+1)! \right] $.
\label{prop_416}
\end{proposition}

\begin{proof} Recall from  (\ref{eq_1102}) above that for any $y \in \RR^{n+1}$,
\begin{equation}  e^{\Phi_V(y)} = \int_V e^{\langle y, x \rangle} dx = \frac{n!}{|y|} \cdot Vol_{n-1}(K_y) = (n+1)! \cdot Vol_n(C_y), \label{eq_144}
\end{equation}
	where $K_y = \left \{ x \in V \, ; \, \langle x,y \rangle = -1 \right \}$ while
\begin{equation} C_y = \left \{ x \in V \, ; \, \langle x,y \rangle \geq -1 \right \}.
\label{eq_910} \end{equation}
	 A point $x \in \RR^{n+1}$ belongs to $V_{\delta}$ if and only if the following holds:
	The point $x$ belongs to $\i(V)$ and for any $y \in \i(V^*)$ with $\langle x, y \rangle = -1$,
	$$ Vol_{n+1}(C_y) \geq \delta. $$
	By homogeneity, we see that for any $x \in \i(V)$,
	$$
	x \in V_{\delta} \quad \Longleftrightarrow \quad \forall y \in \i(V^*), \quad  (- \langle x, y \rangle)^{n+1} \cdot Vol_{n+1}(C_y) \geq  \delta.
	$$
	From (\ref{eq_144}), for any $x \in \i(V)$
	we see that $x \in V_{\delta}$ if and only if
	\begin{equation*}
	(n+1)! \cdot \delta \leq \inf_{y \in \i(V^*)} e^{\Phi_V(y)} \cdot  (- \langle x, y \rangle)^{n+1}  = \left(\frac{n+1}{e} \right)^{n+1} \cdot e^{-\Phi_V^*(x)},
	\end{equation*}
where the last passage is the content of  formula (\ref{eq_958}) above.
\end{proof}

\begin{remark} {\rm
Given a boundary point $x \in \partial V_{\delta}$
we may look at the normal $N(x)$ to the smooth hypersurface $\partial V_{\delta}$ at the point $x$, pointing outwards of $V_{\delta}$, and satisfying
$$ |\langle N(x), x \rangle| = 1. $$
Then $N(x) = \nabla \Phi_V^*(x) / (n+1)$ by
Proposition \ref{prop_416}, and by (\ref{eq_1126}),
 $$
\Phi_V(N(x)) = (n+1) \log \frac{n+1}{e} - \Phi_V^*(x) = \log((n+1)!) + \log \delta. $$
It follows that the {\it polar hypersurface} to $\partial V_{\delta}$, which is defined as the left-hand side of the following formula, satisfies
$$ \{ N(x) \, ; \, x \in \partial V_{\delta} \} = \left \{ y \in V^* \, ; \, \Phi_V(y) = \log( (n+1)!) + \log \delta \right \}. $$
In other words, the level sets of the Laplace transform of the cone $V$
are the polar hypersurfaces to the boundaries of the floating bodies $V_{\delta}$.}
\end{remark}

\medskip
In addition to the convex functions $\Phi_{V^*}$ and $\Phi_V^*$, we shall introduce yet another
convex function that is canonically defined on a proper, convex cone $V$.
It is influenced by Schmuckenschl\"ager's work \cite{Schm}. For a proper, convex cone
$V \subset \RR^{n+1}$ and $x \in \i(V)$ we  define
$$ \Psi_V(x) = -\log (1_V * 1_V)(x) = -\log Vol_{n+1}(V \cap (x - V)), $$
the {\it self-convolution function} of the cone. Here $1_V$ is the
characteristic function of the set $V$, which attains the value $1$ in $V$ and vanishes elsewhere.
Since $V$ is convex,  the convolution $1_V * 1_V$ is a log-concave function by the Brunn-Minkowski inequality
and hence $\Psi_V$ is a convex function which is finite in $\i(V)$. Moreover,
\begin{equation}  \Psi_V(tx) = -(n+1) \log t - \log Vol_{n+1}(V/t \cap (x - V/t)) = -(n+1) \log t + \Psi_V(x). \label{eq_544} \end{equation}
Thus the convex function $\Psi_V$ has the same homogeneity as its sisters $\Phi_{V^*}$ and $\Phi_V^*$.
The convex function $\Psi_V: \i(V) \rightarrow \RR$ does not seem  smooth in general. However, it is certainly smooth when
the boundary of $K_y$ is smooth and strongly convex for some (and hence for all)  $y \in \i(V^*)$. Recall that for $x \in \i(V)$ we denote
$$ T_x = \{ z \in V^* \, ; \, \langle x, z \rangle = -1 \}. $$
We claim that the point $ y= \nabla \Phi_V^*(x) / (n+1)$ is the Santal\'o point of $T_x$. Indeed, since $\nabla \Phi_V^*$ is $(-1)$-homogeneous, $x = \nabla \Phi_V(y) / (n+1)$.
From Lemma
\ref{lem_1122} we know that $x$ is the barycenter of $K_y$.
Consequently $y$ is the Santal\'o point of $T_x$ as  explained in
Remark \ref{rem_1208}.

\begin{proposition} Let $V \subset \RR^{n+1}$ be a proper, convex cone. Then for any $x \in \i(V)$,
	\begin{equation} \Psi_V(x) \geq \Phi_{V}^*(x)  + \kappa_n, \label{eq_637}
	\end{equation}
	where $\kappa_n = \log(2^n (n+1)!) - (n+1) \log \left( \frac{n+1}{e} \right)$. There is equality in (\ref{eq_637})
if and only if $T_x$ is centrally-symmetric with respect to some point in $T_x$.

\smallskip Moreover, consider the case where $T_x$ is centrally-symmetric with respect to some point in $T_x$,
and where the boundary of $T_x$ is smooth and strongly convex. Then the equality in (\ref{eq_637}) holds to first order in $x$, and consequently in this case,
\begin{equation}  \nabla^2 \Psi_V(x) \geq \nabla^2 \Phi_{V}^*(x). \label{eq_933} \end{equation}
\label{cor_551}
\end{proposition}

\begin{proof} Set $y = \nabla \Phi_{V}^*(x) / (n+1)$, the Santal\'o point of $T_x$. Then $\langle x, y \rangle = -1$ as $y \in T_x$.
For any
	 point
	$ z \in V \cap (x - V)$,
	the point $x - z$ also belongs to $V$. Therefore the convex body $V \cap (x - V)$ is centrally-symmetric around the point $x/2$.
	Consequently,
	\begin{equation}  Vol_{n+1} \left( \left \{ z \in V \cap (x - V) \, ; \, \left \langle z - \frac{x}{2}, y \right \rangle \geq 0 \right \}      \right) = \frac{1}{2} \cdot Vol_{n+1}(V \cap (x - V)). \label{eq_454} \end{equation}
	From (\ref{eq_454}) and from the fact that $\langle x, y \rangle = -1$,	
	\begin{equation} Vol_{n+1} \left( \left \{ z \in V \, ; \,   \langle z, y \rangle \geq -\frac{1}{2} \right \} \right) \geq \frac{1}{2} \cdot Vol_{n+1}(V \cap (x - V)). \label{eq_502} \end{equation}
The left-hand side of (\ref{eq_502}) equals $Vol_{n+1}(C_{2y})$ while the right-hand side equals $e^{-\Psi_V(x)} / 2$. Thanks to (\ref{eq_144}) we may rephrase (\ref{eq_502}) as
\begin{equation}  \Phi_V(2y) - \log (n+1)! \geq -\Psi_V(x) - \log 2. \label{eq_920}
\end{equation}
Since  $y = \nabla \Phi_{V}^*(x) / (n+1)$, from properties (\ref{eq_408}) and (\ref{eq_1126}) we obtain
\begin{equation} \Phi_V^*(x) + \Phi_V( 2 y) = (n+1) \log [(n+1)/(2e)]. \label{eq_928} \end{equation}
Now (\ref{eq_637}) follows from (\ref{eq_920}) and (\ref{eq_928}).

\smallskip     Equality in (\ref{eq_637}) is equivalent to equality in (\ref{eq_502}).
If equality holds in (\ref{eq_502}) then the closed convex set
	$V \cap (x - V)$ must contain the entire slice $K_{2y}$, or equivalently,
	$$ K_{2y} \cap (x - K_{2y}) \supseteq K_{2y}. $$
	This means that $K_{2y}$ is centrally-symmetric around the point $x/2 \in K_{2y}$.
This central symmetry condition is not only necessary but it is also sufficient for equality in (\ref{eq_502}),
as it implies that $V \cap (x - V)$ is a double cone with base $K_{2y}$ and apices $0$ and $x$,
which leads to equality in (\ref{eq_502}). We have thus proven that equality holds
in (\ref{eq_637}) if and only if $K_y$ is centrally-symmetric around $x$,
	which according to Remark \ref{rem_1208} happens if and only if $T_x$ is
	centrally-symmetric around $y$. Note that if $T_x$ is centrally-symmetric with respect to some point, then this point must be the Santal\'o point $y$.

\smallskip We move on to the ``Moreover'' part. Assume that $T_x$ has a smooth and strongly convex boundary,
and that it is centrally-symmetric with respect to a certain point $y \in T_x$.
 Then $y \in \i(V^*)$ with
$\langle x, y \rangle = -1$, and $K_y$ is centrally-symmetric around the point $x$.
We thus see that the cone $V$ has a non-trivial symmetry, which is the linear map
$$ S:\RR^{n+1} \rightarrow \RR^{n+1} $$
with $S(x) = x$ and $S(z) = -z$ for any $z \in y^{\perp}$.
Since $S(V) = V$, the functions  $\Phi_V^* \circ S - \Phi_V^*$ and $\Psi_V \circ S - \Psi_V$ are constant in $\i(V)$. It follows that for any $z \in \i(V)$,
$$ S^* \left( \nabla \Phi_V^*(Sz) \right)= \nabla \Phi_V^*(z)
\qquad \text{and} \qquad
S^* \left( \nabla \Psi_V(Sz) \right)= \nabla \Psi_V(z). $$
In particular, at the point $x$, which is a fixed point of the symmetry $S$, the gradients  $\nabla \Phi_V^*(x)$ and $\nabla \Psi_V(x)$ are both fixed points of $S^*$. However, the eigenspace of $S^*$ that corresponds to the eigenvalue one is $1$-dimensional, as $S^*(y) = y$ and $S^*(z) = -z$ for $z \in x^{\perp}$.
Hence the vectors $\nabla \Phi_V^*(x)$ and $\nabla \Psi_V(x)$ are proportional.
The homogeneity relations (\ref{eq_408}) and (\ref{eq_544}) yield
$$ \langle \nabla \Psi_V(x), x \rangle = -(n+1) =
\langle \nabla \Phi_V^*(x), x \rangle. $$
Since the gradients are proportional, then necessarily $\nabla \Psi_V(x) = \nabla \Phi_V^*(x)$,
and the equality in (\ref{eq_637}) is to first order. The Hessian inequality (\ref{eq_933}) follows.
 \end{proof}

The following is a crude reverse form of Proposition \ref{cor_551}:

\begin{proposition} Let $V \subset \RR^{n+1}$ be a proper, convex cone. Then for any $x \in \i(V)$,
	\begin{equation} \Psi_V(x) \leq \Phi_{V}^*(x)  + C n, \label{eq_939}
	\end{equation} \label{lem_822}
where $C > 0$ is a universal constant.\end{proposition}

\begin{proof} Set $y = \nabla \Phi_{V}^*(x) / (n+1)$, the Santal\'o point of $T_x$. Then $\b(K_y) = x$ by Lemma \ref{lem_1122}. By Fubini's theorem,
$$ Vol_{n+1}(V \cap (x - V)) = \frac{1}{|y|} \int_0^1 Vol_n[ t K_y \cap (x - (1-t) K_y)] dt.
 $$
The barycenter of $K_y-x$ lies at the origin, and by Milman and Pajor \cite[page 321]{MiP2},
$$
Vol_n[ t K_y \cap (x - (1-t) K_y)] = Vol_n[( t (K_y -x)   \cap  (1-t) ( x - K_y) ]
\geq t^n (1-t)^n Vol_n(K_y). $$
Therefore,
\begin{equation} e^{-\Psi_V(x)} = Vol_{n+1}(V \cap (x - V)) \geq \frac{Vol_{n}(K_y)}{|y|} \int_0^1 t^n (1-t)^n dt = \frac{n! \cdot e^{\Phi_V(y)}}{(2n+1)!} ,
\label{eq_1000}
\end{equation}
where we used (\ref{poly}) in the last passage. As in (\ref{eq_928}) above, we know that
\begin{equation}  \Phi_V^*(x) + \Phi_V( y) = (n+1) \log [(n+1)/e]. \label{eq_1001_}
\end{equation}
Now (\ref{eq_939}) follows from (\ref{eq_1000}), (\ref{eq_1001_}) and the fact that $(2n+1)! \cdot e^{n+1} \leq C^n \cdot n! \cdot (n+1)^{n+1}$.
\end{proof}

\section{The isomorphic slicing problem}

In this section we prove Theorem \ref{thm_903}. We begin with a formula for  the isotropic constant of a hyperplane section of $V$.
Recall that for any $A \in \RR^{(n+1) \times (n+1)}$ and $v \in \RR^{n+1}$,
$$ \det(A + v v^*) = \det(A) + v^* {\rm Adj}(A) v  $$
where ${\rm Adj}(A)$ is the adjoint matrix.

\begin{lemma} For any proper, convex cone $V \subset \RR^{n+1}$ and $y \in \i(V^*)$,
\begin{equation}  \det \nabla^2 \Phi_V(y) = \kappa_n \cdot L_{K_y}^{2n} \cdot e^{ 2 \Phi_V(y)}
\label{eq_455}
\end{equation}
with $\kappa_n = (n+1)^{n+1} \cdot (n+2)^n / (n!)^2$.  \label{lem_208}
\end{lemma}

\begin{proof} By Lemma \ref{lem_1122},
$$ \frac{\nabla^2 \Phi_V(y)}{(n+1) (n+2)} =  \cov(K_y) + \frac{b(K_y) b^*(K_y)}{n+2}. $$
The symmetric matrix $\cov(K_y)$ is of rank $n$, with the vector $y$ spanning its kernel. Therefore the adjoint matrix of $\cov(K_y)$ is
$$ {\rm det}_n \cov(K_y) \cdot \frac{y y^*}{|y|^2}, $$
where ${\rm det}_n(A)$ stands for the sum of the determinants    of all principal $n \times n$ minors of a matrix $A \in \RR^{(n+1) \times (n+1)}$.
Consequently,
$$ \det \nabla^2 \Phi_V(y) = (n+1)^{n+1} \cdot (n+2)^{n+1} \cdot \frac{{\rm det}_n \cov(K_y)}{|y|^2} \cdot \frac{\langle y, \b(K_y) \rangle^2}{n+2}. $$
However, $\langle \b(K_y), y \rangle = -1$
as $\b(K_y) \in K_y = \{ x \in V \, ; \, \langle x,y \rangle = -1 \}$. Hence, by (\ref{poly}),
$$ L_{K_y}^{2n} = \frac{{\rm det}_n \cov(K_y)}{Vol_n(K_y)^2} = \frac{(n!)^2}{(n+1)^{n+1} \cdot (n+2)^n} \cdot \frac{\det \nabla^2 \Phi_V(y)}{e^{2 \Phi_V(y)}},  $$
and the formula follows.
\end{proof}

The role of the determinant $\nabla^2 \Phi_V$ is twofold: First, it appears in the expression
for the isotropic constant in Lemma \ref{lem_208}. Second, it is the Jacobian determinant
of the diffeomorphism $\nabla \Phi_V: \i(V^*) \rightarrow \i(V)$.
The next lemma describes a certain
geometric property of this map. Recall
that for $y \in \i(V^*)$ we denote $C_y = \left \{ x \in V \, ; \, \langle x, y \rangle \geq -1 \right \}$, a truncated cone.

\begin{lemma} Let $V \subset \RR^{n+1}$ be a proper, convex cone and let $y \in \i(V^*)$.
	Then,
	$$ \nabla \Phi_V( y + V^*) \subseteq (n+1) \cdot C_y. $$
\label{lem_1120}
\end{lemma}

\begin{proof} It suffices
	to prove that the image of the open set $y + \i(V^*)$ under the diffeomorphism $\nabla \Phi_V$ is contained in the closed set $(n+1) C_y$. Let  $z \in y + \i(V^*)$. Then for any $x \in K_y$
we have $ \langle x, z - y \rangle < 0$. This means that $$ K_y \cap K_z = \emptyset, $$ as there is no point $x \in K_y$ with $\langle x, z \rangle =
\langle x, y \rangle = -1$. The convex set $K_y$ disconnects the cone $V$ into two connected components. The set $K_z$ must be contained in the convex set $C_y$, and hence its barycenter satisfies
$ \b(K_z) \in C_y$. By Lemma \ref{lem_1122}, we know that $\nabla \Phi_V(z) = (n+1) \cdot \b(K_z)$, and the conclusion of the lemma follows.
\end{proof}

\begin{remark} {\rm From the proof of Lemma \ref{lem_1120} we obtain a simple geometric interpretation of the set $\nabla \Phi_V( y + \i(V^*))$.
		Namely, this set consists of all barycenters of all hyperplane sections of the truncated cone $(n+1) \cdot C_y$ that are disjoint from the base of this truncated cone. Additionally, a simple modification of the proof of Lemma \ref{lem_1120} shows that
		$$ \nabla \Phi_V( y - V^*) \subseteq V \setminus (n+1) C_y. $$
} \end{remark}

\begin{lemma} Let $V \subset \RR^{n+1}$ be a proper, convex cone, $y_0 \in \i(V^*)$ and let $0 < \eps < 1$. Then there exists
a point $y \in \i(V^*)$ such that $y - y_0 \in V^* \cap (\eps y_0 - V^*)$ and
\begin{equation} L_{K_{y}} \leq C / \sqrt{\eps}
 \label{eq_154_} \end{equation}
where $C > 0$ is a universal constant. \label{lem_413}
\end{lemma}

\begin{proof}
In this proof $C, \tilde{C}, \bar{C}, \hat{C} > 0$ denote various positive universal constants, whose value may change from one line to the next.
We may assume that $\eps > e^{-n}$ since otherwise conclusion (\ref{eq_154_}) follows from
	the trivial upper bound $L_{K_{y_0}} \leq C \sqrt{n}$ (see, e.g., \cite{BGVV}).
Define
$$ S = (y_0 + V^*) \cap (y_0 + \eps y_0 - V^*). $$
According to Proposition \ref{lem_822},
\begin{equation}  Vol_{n+1}(S) = Vol_{n+1}( V^* \cap (\eps y_0 - V^*) ) = \exp(-\Psi_{V^*}(\eps y_0)) \geq \exp(-C n - \Phi_{V^*}^*(\eps y_0)).
\label{eq_840}
\end{equation}
We would like to get rid of the expression $\Phi_{V^*}^*(\eps y_0)$, and replace it by $\Phi_V(\eps y_0)$ plus an error term.
To this end, we may use the ``commutation relation'' of Corollary \ref{lem_842}, according to which for any $y \in \i(V^*)$,
\begin{equation}  \Phi_V(y) - \Phi_{V^*}^*(y) = \log \frac{(n!)^2 \cdot e^{n+1}}{(n+1)^{n+1}}  + \log \bar{s}(K_y) \geq n \log n - C n + \log \bar{s}(K_y). \label{eq_841}
\end{equation}
However, from the Bourgain-Milman inequality
(\ref{eq_531}), we know that $\bar{s}(K_y) \geq (c / n)^n$ for some universal constant $c > 0$. Hence, from (\ref{eq_840}) and
(\ref{eq_841}),
\begin{equation}
Vol_{n+1}(S) \geq \exp(-\tilde{C} n - \Phi_V(\eps y_0)).
\label{eq_842} \end{equation}
For any $y \in S$ we know that $y - (1 + \eps) y_0 \in -V^*$, and hence the scalar product of $y - (1 + \eps) y_0$ with $\nabla \Phi_V(y_0 + \eps y_0) \in V$
is non-negative. By the convexity of $\Phi_V$, for any $y \in S$,
\begin{equation} \Phi_V(y)   \geq \Phi_V( y_0 + \eps y_0 ) + \langle \nabla \Phi_V( y_0 + \eps y_0 ), y - (1+\eps) y_0 \rangle \geq \Phi_V(y_0 + \eps y_0).
\label{eq_850_} \end{equation}
From (\ref{eq_842}) and (\ref{eq_850_})
\begin{equation} \int_S e^{2 \Phi_V(y)} dy \geq e^{-\tilde{C} n} \cdot e^{2 \Phi_V( (1 + \eps) y_0 ) - \Phi_V(\eps y_0) }
= e^{-\tilde{C} n} \cdot  \left( \frac{\eps}{(1 + \eps)^2} \right)^{n+1} \cdot  e^{\Phi_V(y_0)},
 \label{eq_219} \end{equation}
 where we used the homogeneity relation (\ref{eq_408}) in the last psssage.
According to Lemma \ref{lem_1120}, the set $\nabla \Phi_V(S)$ is contained in the truncated cone $(n+1) \cdot C_{y_0}$.
Corollary \ref{cor_diffeo} states that the map
$\nabla \Phi_V$ is a diffeomorphism. By changing variables,
\begin{equation} Vol_{n+1}((n+1) \cdot C_{y_0}) \geq Vol_{n+1} \left( \nabla \Phi_V(S) \right) = \int_S \det \nabla^2 \Phi_V(y) dy = \kappa_n \int_S e^{2 \Phi_V(y)} \cdot L_{K_y}^{2n} dy,
\label{eq_125_}
\end{equation}
where $\kappa_n = (n+1)^{n+1} \cdot (n+2)^n / (n!)^2 \geq 1$ is the coefficient  from Lemma \ref{lem_208}. Recall
the formula $Vol_{n+1}(C_{y_0}) = \exp(\Phi_V(y_0)) / (n+1)!$ according to (\ref{eq_144}). We thus deduce from (\ref{eq_125_}) that
\begin{equation} \frac{(n+1)^{n+1}}{(n+1)!}
e^{\Phi_V(y_0)} \geq \inf_{y \in S} L_{K_y}^{2n} \cdot  \int_S e^{2 \Phi_V(y)} dy.
\label{eq_1251_}
\end{equation}
From (\ref{eq_219}) and (\ref{eq_1251_}),
$$  \frac{(n+1)^{n+1}}{(n+1)!} \cdot e^{\Phi_V(y_0)} \geq \inf_{y \in S} L_{K_y}^{2n} \cdot  e^{-\tilde{C} n} \cdot  \left( \frac{\eps}{(1 + \eps)^2} \right)^{n+1} \cdot  e^{\Phi_V(y_0)}. $$
This implies that there exists $y \in S$ for which
$$ L_{K_y}^{2n} \leq \left( \frac{\hat{C}}{\eps  } \right)^{n+1}
\leq \left( \frac{\bar{C}}{\eps } \right)^{n},
$$ where we used the assumption that $\eps > e^{-n}$ in the last passage. This completes the proof  of the lemma.
\end{proof}

Given a convex body $K \subseteq \RR^n$ with the origin in its interior, we consider the associated (non-symmetric) norm
$$ \| x \|_K = \inf \{ \lambda \geq 0 \, ; \, x \in \lambda K \} \qquad \qquad (x \in \RR^n). $$
The supporting functional of $K$ is
$$ h_K(y) = \| y \|_{K^{\circ}} = \sup_{z \in K} \langle y,z \rangle = \sup_{0 \neq z \in \RR^n} \frac{\langle y, z \rangle}{\| z \|_K}\qquad \qquad (y \in \RR^n). $$
We write $\pi(t,x) = x$ for $(t,x) \in \RR \times \RR^n \cong \RR^{n+1}$.

\begin{lemma} Let $K \subseteq \RR^n$ be a convex body with the origin in its interior,
and set
\begin{equation}
V = \left \{ (t, tx) \in \RR \times \RR^n \, ; \, t \geq 0, x \in K \right \}. \label{eq_854}
\end{equation}
Then for any $y \in \i(V^*)$ and $x \in \RR^n$, denoting $y = (y_1, \pi(y)) \in \RR \times \RR^n$ we have
\begin{equation}
\| x \|_{\pi(K_y)} = -y_1 \| x \|_K - \langle x, \pi(y) \rangle,
\label{eq_910_}
\end{equation}
and when setting $T = \pi(K_y) \subseteq \RR^n$ we obtain
\begin{equation}
T^{\circ} = - y_1 K^{\circ} -\pi(y).
\label{eq_928_} \end{equation} \label{lem_435}
\end{lemma}

\begin{proof} By definition,
	$$ K_y = \left \{ (t, tx) \, ; \, t \geq 0, \ x \in K, \ t y_1 + \langle tx, \pi(y) \rangle  = -1 \right \}. $$
It follows that
$$ \pi(K_y) = \left \{ x \in \RR^n \, ; \, \| x \|_K \leq -\frac{1+ \langle x, \pi(y) \rangle}{y_1} \right \}, $$
from which (\ref{eq_910_}) follows. Next, for any $z \in \RR^n$, we see that $z \in T^{\circ}$
if and only if
\begin{equation}
 \langle z, x \rangle \leq \| x \|_{\pi(K_y)} = -y_1 \| x \|_K - \langle x, \pi(y) \rangle
\qquad \textrm{for all} \ x \in \RR^n. \label{eq_924} \end{equation}
Condition (\ref{eq_924}) is  equivalent to $\langle z + \pi(y) , x \rangle \leq -y_1 \| x \|_K$
for all $x \in \RR^n$. Thus $z \in T^{\circ}$ if and only if $h_K( z + \pi(y)  )  \leq -y_1$,
or equivalently if and only if $z + \pi(y) \in -y_1 K^{\circ}$. The relation  (\ref{eq_928_}) follows.
\end{proof}

\begin{proof}[Proof of Theorem \ref{thm_903}:] Define $V$ as in (\ref{eq_854}).
Since the barycenter of $K$ lies at the origin, the point $e = (1, 0) \in \RR \times \RR^n$
is the barycenter of $K_{-e} = \{1 \} \times K$. We will apply Lemma \ref{lem_413} with $y_0 = -e$. From the conclusion of this lemma, there exists $y \in \i(V^*)$  such that
	$$ L_{K_{y}} \leq C / \sqrt{\eps}.
	$$
Moreover, $y + e  \in V^* \cap (-\eps e - V^*)$.
Since $\langle z, e \rangle \leq 0$ for any $z \in V^*$, by setting $y_1 := \langle y, e \rangle$ we have
\begin{equation} -1-\eps \leq y_1 \leq -1. \label{eq_1136} \end{equation}
Since $y + e \in V^*$ we obtain from (\ref{eq_1136}) that
$$ y + e \in \left \{ z \in V^* \, ; \, -\eps \leq z_1 \leq 0 \right \}
= \left \{ (t, t w ) \in \RR \times \RR^n \, ; \,
-\eps \leq t \leq 0,  \, w \in -K^{\circ} \right \}. $$
Thus $\pi(y) \in \eps K^{\circ}$. Since $y + e \in -\eps e - V^*$ we obtain from (\ref{eq_1136}) that
$$ y + e \in \left \{ z \in -\eps e - V^* \, ; \, -\eps \leq z_1 \leq 0 \right \}
= \left \{ (-\eps+t, t w ) \in \RR \times \RR^n \, ; \,
0 \leq t \leq \eps,  \, w \in -K^{\circ} \right \}. $$
Thus $\pi(y) \in -\eps K^{\circ}$. To summarize,
\begin{equation} \pi(y) \in \eps (K^{\circ} \cap (-K^{\circ}) ). \label{eq_1153} \end{equation}
Denote
$$ T = -y_1 \cdot \pi(K_y). $$
Then $T \subseteq \RR^n$ is a convex body that is affinely equivalent to $K_y$, and hence $L_T = L_{K_y} < C / \sqrt{\eps}$.
Furthermore, by Lemma \ref{lem_435},
\begin{equation}  T^{\circ} = K^{\circ} + \frac{\pi(y)}{y_1}, \label{eq_1154} \end{equation}
and $T^{\circ}$ is a translate of $K^{\circ}$. From (\ref{eq_1136}) and (\ref{eq_1153}) we know that $\pi(y) / y_1$ belongs
to $\eps (-K^{\circ}) \cap \eps K^{\circ}$. Hence, from (\ref{eq_1154}),
$$ T^{\circ} \subseteq K^{\circ} + \eps K^{\circ} \qquad \text{and} \qquad K^{\circ} \subseteq T^{\circ} + \eps K^{\circ}. $$
Equivalently,
$$ (1 - \eps) K^{\circ} \subseteq  T^{\circ} \subseteq (1 + \eps) K^{\circ}. $$
Since $0 < \eps < 1/2$, by dualizing this inclusion we obtain
$$ (1 - \eps) \cdot K \subseteq \frac{1}{1+\eps} \cdot K \subseteq  T \subseteq \frac{1}{1-\eps} \cdot K \subseteq (1 + 2 \eps) \cdot K, $$
and the theorem follows by adjusting the universal constant $C$.
\end{proof}

{
}

\smallbreak
\noindent Department of Mathematics, Weizmann Institute of Science, Rehovot 76100 Israel, and
School of Mathematical Sciences, Tel Aviv University, Tel Aviv 69978 Israel.

\smallbreak
\hfill \verb"boaz.klartag@weizmann.ac.il"

\end{document}